\documentclass[11pt,a4paper]{article}
\usepackage{indentfirst}
\usepackage{amsfonts}
\usepackage{amssymb}
\usepackage{mathrsfs}
\usepackage{amsmath}
\usepackage{amsthm}
\usepackage{enumerate}
\usepackage{cite}
\usepackage{mathrsfs}
\allowdisplaybreaks
\usepackage{geometry}
\usepackage{ifpdf}
\ifpdf
  \usepackage[colorlinks=true,linkcolor=blue,citecolor=red, final,backref=page,hyperindex]{hyperref}
\else
  \usepackage[colorlinks,final,backref=page,hyperindex,hypertex]{hyperref}
\fi

\usepackage{txfonts}
\usepackage{indentfirst}
\usepackage{latexsym}
\usepackage[all]{xy}
\def\<{\langle}
\def\>{\rangle}

\newtheorem{thm}{Theorem}[section]
\newtheorem{lem}[thm]{Lemma}

\newtheorem{pro}[thm]{Proposition}
\newtheorem{ex}[thm]{Example}

\theoremstyle{definition}
\newtheorem{defi}{Definition}[section]

\theoremstyle{remark}
\newtheorem{rmk}{Remark}[section]
\begin{document}
\title{\bf  Twisted $\mathcal{O}$-operators on $3$-Lie
algebras and $3$-NS-Lie algebras}
\author{\bf T. Chtioui, A. Hajjaji, S. Mabrouk, A. Makhlouf}
\author{{ Taoufik Chtioui$^{1}$
 \footnote {  E-mail: chtioui.taoufik@yahoo.fr}
,\  Atef Hajjaji$^{1}$
    \footnote {  E-mail:  atefhajjaji100@gmail.com}
,\  Sami Mabrouk$^{2}$
 \footnote {   E-mail: mabrouksami00@yahoo.fr}
\ and Abdenacer Makhlouf$^{3}$
 \footnote { E-mail: Abdenacer.Makhlouf@uha.fr}
}\\
{\small 1.  University of Sfax, Faculty of Sciences Sfax,  BP
1171, 3038 Sfax, Tunisia} \\
{\small 2.  University of Gafsa, Faculty of Sciences Gafsa, 2112 Gafsa, Tunisia}\\
{\small 3.~ IRIMAS - Département de Mathématiques, 6, rue des frères Lumière,
F-68093 Mulhouse, France}}
\date{}
\maketitle
\begin{abstract}
The purpose of this paper is to  introduce twisted $\mathcal{O}$-operators on  $3$-Lie algebras. We define a cohomology
of a twisted $\mathcal{O}$-operator $T$ as the Chevalley-Eilenberg cohomology of a certain $3$-Lie algebra induced by $T$ with coefficients in
a suitable representation. Then we consider infinitesimal  and formal deformations of twisted
$\mathcal{O}$-operators from cohomological points of view. Furthermore,  we
introduce and study $3$-NS-Lie- algebras as the underlying structure of twisted
$\mathcal{O}$-operators on $3$-Lie algebras. Finally, we investigate    twisted $\mathcal{O}$-operators on $3$-Lie algebras induced by Lie algebras.  
\end{abstract}



\numberwithin{equation}{section}

\tableofcontents

%%%%%%%%%%%%%%%%%%%%%%%%%%%%%%%%%%%%%
\section{Introduction} 

%%%%%%%%%%%%%%%%%%%%%%%%%%%%%%%%%%%%%%
 A natural generalization of binary operations appeared first when Cayley studied  cubic matrices which are ternary operations.  Furthermore one may consider in general $n$-ary operations of associative type or Lie type.  In particular,
$3$-Lie algebras
and more
generally, $n$-Lie algebras  \cite{filippov}  are  generalizations of Lie algebras to ternary and $n$-ary cases. Recall that a
$3$-Lie algebra 
 $\mathfrak{g}$ is a vector space together with a skew-symmetric 3-linear
map $[\cdot,\cdot,\cdot]_\mathfrak{g}:\wedge^{3}\mathfrak{g}\rightarrow \mathfrak{g},$ such that for $x_i\in \mathfrak{g}, 1\leq i\leq 5$, the following Filippov-Jacobi Identity (sometimes called fundamental identity or Nambu identity)
holds
\begin{align}\label{eq:de1}
[x_1,x_2,[x_3,x_4,x_5]_\mathfrak{g}]_\mathfrak{g}=[[x_1,x_2,x_3]_\mathfrak{g},x_4,x_5]_\mathfrak{g}+[x_3,[x_1,x_2,x_4]_\mathfrak{g},x_5]_\mathfrak{g}+[x_3,x_4,[x_1,x_2,x_5]_\mathfrak{g}]_\mathfrak{g}.
\end{align}

The first instances of ternary Lie algebras are related to Nambu Mechanics  \cite{nambu}, which was formulated  algebraically by Takhtajan \cite{Takhtajan}. The first complete algebraic study of $n$-Lie algebras is due to Filippov \cite{filippov}.  We refer  to \cite{AMS, BBW, BSZ}  for the realizations
and classifications of $3$-Lie algebras and $n$-Lie algebras. Ternary operations turn to be useful in many Mathematics and Physics domains, like string theory. The quantization of the  Nambu brackets in \cite{Awata} was a motivation to present a general construction  of $(n+1)$-Lie algebras induced by $n$-Lie algebras using  the $n$-ary  brackets  and trace-like linear forms, see \cite{AMS0,AMS,Kitouni}. The structure of $3$-Lie (super)algebras induced by Lie (super)algebras, classification of $3$-Lie algebras and application to constructions of B.R.S. algebras have been considered in \cite{Abramov2018:WeilAlg3LiealgBRSalg,AbramovLatt2016:classifLowdim3Liesuperalg,Abramov2017:Super3LiealgebrasinducedsuperLiealg}.

 A deformation theory based on one-parameter formal power series Deformations  was  introduced first  by Gerstenhaber  \cite{Gerst}
for associative algebras and then extends to Lie algebras by   Nijenhuis and Richardson in \cite{NR}. It is shown that deformations are controlled by suitable cohomologies, Hochschild cohomology in associative case and Chevalley-Eilenberg cohomology in Lie case.  The same approach was used for various algebraic structures. Deformations of  3-Lie algebras were studied in  \cite{Fig}. Cohomology is a main tool for deformation theory and provides also invariants to study  algebraic structures. 

In\cite{BGS}, the authors studied  the solutions of 3-Lie classical Yang-Baxter equation, that lead to introduce  the notion of $\mathcal{O}$-operator on 3-Lie algebras with respect to a representation. In particular,  Rota-Baxter operators on  3-Lie algebras, introduced in \cite{BGLW},   are  $\mathcal{O}$-operators on a 3-Lie algebra with respect to the adjoint representation. 

Twisted Rota-Baxter operators introduced by Uchino  in the context of associative algebras  \cite{Uchino}  are algebraic analogue of  twisted Poisson structure  introduced and studied in  \cite{SW,KS}. They are also related to NS-algebras considered by  Leroux  \cite{Leroux}.
Twisted Rota-Baxter operators on Lie algebras and Leibniz algebras were studied in \cite{Das1,Das3}. A cohomology of twisted Rota-Baxter operator was derived, in \cite{Das1,Das2}, from a suitable $L_\infty$-algebra whose Maurer-Cartan elements are given
by twisted Rota-Baxter operators.  Such a cohomology can be seen as the Hochschild (resp. Chevalley-Eilenberg) cohomology of a certain Lie algebra with coefficients in a suitable representation.  A cohomology
of a twisted relative Rota-Baxter operator  as the Loday-Pirashvili
cohomology of a certain Leibniz algebra was constructed in  \cite{Das3}. One may see also  in \cite {CaiSheng,LazarevSheng} for  Hom version of Nijenhuis Bracket and  cohomologies of relative Rota-Baxter Lie algebras. 

The main purpose of this paper is to study twisted $\mathcal{O}$-operators on 3-Lie algebras. We provide some characterizations and key constructions. We define a cohomology  of  twisted $\mathcal{O}$-operators that controls their deformations.  Moreover, we introduce a new  ternary  algebraic structure called 3-NS-Lie algebras that are related to
twisted $\mathcal{O}$-operators in the same way as 3-pre-Lie algebras are related to $\mathcal{O}$-operators. Finally we provide a relationship between the theories for binary structures and  ternary structures through the approach using trace-like maps.

The paper is organized as follows. In Section 2, we briefly recall basics about representations and cohomology of 3-Lie algebras. Then
 we introduce $\Theta$-twisted $\mathcal{O}$-operators on 3-Lie algebras,  provide some examples and characterization results.
Section 3 is devoted to define the
cohomology of a  twisted $\mathcal{O}$-operator on a 3-Lie algebra using the underlying
3-Lie algebra of a twisted $\mathcal{O}$-operator.
In Section 4, we study deformations
of twisted $\mathcal{O}$-operators and show that they are controlled by  the cohomology theory established in
Section 3. In Section 5   we introduce 3-NS-Lie algebras which are derived naturally from   twisted $\mathcal{O}$-operators. In the last section,  we investigate    twisted $\mathcal{O}$-operators on $3$-Lie algebras induced by Lie algebras along trace-like maps and construct  $3$-NS-Lie algebras from NS-Lie algebras.\\ 

In this paper, we work over an algebraically closed field $\mathbb{K}$ of characteristic $0$ and all the vector spaces are over $\mathbb{K}$.

%%%%%%%%%%%%%%%%%%%%%%%%%%%%%%%%%%%
\section{Twisted $\mathcal{O}$-operators on $3$-Lie algebras}
%%%%%%%%%%%%%%%%%%%%%%%%%%%%%%%%%%%
In this section, we recall some basic definitions about representation, $\mathcal O$-operators  and  cohomology of $3$-Lie algebras, see \cite{Loday,filippov, Kasymov,Takhtajan,Zhang}), and 
 introduce twisted $\mathcal{O}$-operators on $3$-Lie algebras.  Moreover,
we give some constructions of twisted $\mathcal{O}$-operators and provide  examples.

 Let $(\mathfrak{g},[\cdot,\cdot,\cdot]_\mathfrak{g})$ be a $3$-Lie algebra, $V$ be a vector space and $\rho : \wedge^{2}\mathfrak{g}\rightarrow gl(V)$ be a linear map. The pair $(V,\rho)$ is called a representation (or $V$ is a $\mathfrak g$-module) of $\mathfrak{g}$  if $\rho$ satisfies for all $x_1, x_2, x_3, x_4\in \mathfrak{g}$,
\begin{align}\label{eq:rep}
&\rho(x_1,x_2)\rho(x_3,x_4)=\rho([x_1,x_2,x_3]_\mathfrak{g},x_4)+\rho(x_3,[x_1,x_2,x_4]_\mathfrak{g})+\rho(x_3,x_4)\rho(x_1,x_2),\\
&\rho([x_1,x_2,x_3]_\mathfrak{g},x_4)=\rho(x_1,x_2)\rho(x_3,x_4)+\rho(x_2,x_3)\rho(x_1,x_4)+\rho(x_3,x_1)\rho(x_2,x_4).
\end{align}
Set
$\mathfrak{C}_{3Lie}^{n}(\mathfrak{g};V)=Hom(\underbrace{\wedge^{2}\mathfrak{g}\otimes \cdots \otimes \wedge^{2}\mathfrak{g}}_{n-1}\wedge \mathfrak{g},V),\;(n\geq 1)$
 the space of $n$-cochains.

 Consider  the differential  $\partial: \mathfrak{C}_{3Lie}^{n}(\mathfrak{g};V)\rightarrow \mathfrak{C}_{3Lie}^{n+1}(\mathfrak{g};V)$ defined by
\begin{align}
&(\partial f)(x_1,x_2,\cdots,x_{2n+1})\nonumber\\
&=(-1)^{n+1}\rho(x_{2n+1},x_{2n-1})f(x_1,x_2,\cdots,x_{2n-2},x_{2n})\nonumber\\
&+(-1)^{n+1}\rho(x_{2n},x_{2n+1})f(x_1,x_2,\cdots,x_{2n-1})\nonumber\\
&+\sum_{k=1}^{n}(-1)^{k+1}\rho(x_{2k-1},x_{2k})f(x_1,x_2,\cdots,\widehat{x_{2k-1}},\widehat{x_{2k}},\cdots,x_{2n+1})\nonumber\\
&+\sum_{k=1}^{n}\sum_{j=2k+1}^{2n+1}(-1)^{k}f(x_1,x_2,\cdots,\widehat{x_{2k-1}},\widehat{x_{2k}},\cdots,[x_{2k-1},x_{2k},x_j]_\mathfrak{g},\cdots,x_{2n+1})
\end{align}
for all $x_1,x_2,...,x_{2n+1}\in \mathfrak{g}$.
 Thus $(\oplus_{n=1}^{+\infty}\mathfrak{C}_{3Lie}^{n}(\mathfrak{g};V),\partial)$ is a cochain complex which is called Chevalley-Eilenberg cochain complex of 3-Lie algebras.

The quotient space $H_{3Lie}^{n}(\mathfrak{g};V)=Z_{3Lie}^{n}(\mathfrak{g};V)\diagup B_{3Lie}^{n}(\mathfrak{g};V)$, where $Z_{3Lie}^{n}(\mathfrak{g};V)=\{f\in \mathfrak{C}_{3Lie}^{n}(\mathfrak{g};V)|\;\partial f = 0\}$ is the space of $n$-cocycles and $B_{3Lie}^{n}(\mathfrak{g};V)=\{f =
\partial g|\;g\in \mathfrak{C}_{3Lie}^{n-1}(\mathfrak{g};V)\}$ is the space of $n$-coboundaries, is called the cohomology group of the $3$-Lie algebra $\mathfrak{g}$ with coefficients in $V$.

 Let
$\Theta\in \mathfrak{C}_{3Lie}^{2}(\mathfrak{g};V)$ be a $2$-cocycle in the Chevalley-Eilenberg cochain complex. That is  $\Theta : \wedge^{3}\mathfrak{g} \rightarrow V$
is a trilinear map satisfying
\begin{align}\label{eq:cocycle}
&\Theta(x_1,x_2,[y_1,y_2,y_3]_\mathfrak{g})+\rho(x_1,x_2)\Theta(y_1,y_2,y_3)
-\Theta([x_1,x_2,y_1]_\mathfrak{g},y_2,y_3)\nonumber\\
&-\Theta(y_1,[x_1,x_2,y_2]_\mathfrak{g},y_3)-\Theta(y_1,y_2,[x_1,x_2,y_3]_\mathfrak{g})-
\rho(y_2,y_3)\Theta(x_1,x_2,y_1)\nonumber\\
&-\rho(y_3,y_1)\Theta(x_1,x_2,y_2)-\rho(y_1,y_2)\Theta(x_1,x_2,y_3)=0
\end{align}
for $x_1,x_2,y_1,y_2,y_3\in \mathfrak{g}$.
Under the above notations,  the direct sum $\mathfrak{g}\oplus V$ carries a $3$-Lie algebra structure given by
\begin{equation}
[(x,u),(y,v),(z,w)]_{\Theta}=\Big([x,y,z]_\mathfrak{g},\rho(x,y)w+\rho(z,x)v+\rho(y,z)u
+\Theta(x,y,z)\Big),
\end{equation}
which is called the $\Theta$-twisted semi-direct product, denoted by $\mathfrak{g}\ltimes_{\Theta}V$.
\begin{defi}
A linear map $T: V\rightarrow \mathfrak{g}$ is said to be a $\Theta$-twisted $\mathcal{O}$-operator if $T$ satisfies
\begin{equation}\label{eq:twisted}
[Tu,Tv,Tw]_\mathfrak{g}=T\Big(\rho(Tu,Tv)w+\rho(Tv,Tw)u+\rho(Tw,Tu)v+\Theta(Tu,Tv,Tw)\Big)
\end{equation}
for all $u,v,w\in V$.
\end{defi}
Using the twisted
semi-direct product, one can characterize twisted $\mathcal{O}$-operators by their
graphs.
\begin{pro}\label{graph}
A linear map $T:V\rightarrow \mathfrak{g}$ is a $\Theta$-twisted $\mathcal{O}$-operator if and only if the graph $Gr(T)=\{(Tu,u)|\;u\in V\}$ is a subalgebra of the
$\Theta$-twisted semi-direct product $\mathfrak{g}\ltimes_{\Theta}V$.
\end{pro}
\begin{proof}
Let $(Tu,u)$, $(Tv,v)$ and $(Tw,w)$ $\in Gr(T)$. Then we have
\begin{align*}
&\quad \;[(Tu,u),(Tv,v),(Tw,w)]_{\Theta}\\
&=\Big([Tu,Tv,Tw]_\mathfrak{g},\rho(Tu,Tv)w+\rho(Tw,Tu)v+\rho(Tv,Tw)u+\Theta(Tu,Tv,Tw)\Big).
\end{align*}
 Assume that  $Gr(T)$ is a subalgebra of the
$\Theta$-twisted semi-direct product $\mathfrak{g}\ltimes_{\Theta}V$, then  we have
\begin{align*}
&[Tu,Tv,Tw]_\mathfrak{g}=
T\Big(\rho(Tu,Tv)w+\rho(Tw,Tu)v+\rho(Tv,Tw)u+\Theta(Tu,Tv,Tw)\Big).
\end{align*}
On the other hand, if $T$ be a $\Theta$-twisted $\mathcal{O}$-operator. Then we obtain
\begin{align*}
&\quad \;[(Tu,u),(Tv,v),(Tw,w)]_{\Theta}\\
&=\Big(T\big(\rho(Tu,Tv)w+\rho(Tw,Tu)v+\rho(Tv,Tw)u+\Theta(Tu,Tv,Tw)\big),\\
&\quad \;\rho(Tu,Tv)w+\rho(Tw,Tu)v+\rho(Tv,Tw)u+\Theta(Tu,Tv,Tw)\Big) \in Gr(T).
\end{align*}
Hence $Gr(T)$ is a subalgebra of the
$\Theta$-twisted semi-direct product $\mathfrak{g}\ltimes_{\Theta}V$.
\end{proof}
Since $Gr(T)$ is isomorphic to $V$ as a vector space  by the identification $(T(u),u)\cong u$. A
   $\Theta$-twisted $\mathcal{O}$-operator $T$ induces a  $3$-Lie algebra structure on $V$ given by
\begin{equation}\label{eq:cro1}
[u,v,w]_T=\rho(Tu,Tv)w+\rho(Tv,Tw)u+\rho(Tw,Tu)v+\Theta(Tu,Tv,Tw).
\end{equation}
It is obvious that $T$ is an algebra morphism, that is $T([u,v,w]_T)=[Tu,Tv,Tw]_{\mathfrak{g}}$.
\begin{defi}
Let $T:V\rightarrow \mathfrak{g}$ be a $\Theta$-twisted $\mathcal{O}$-operator and  $T':V'\rightarrow \mathfrak{g}'$ be a $\Theta'$-twisted $\mathcal{O}$-operator. A morphism of twisted $\mathcal{O}$-operators from $T$ to $T'$ consists of a pair $(\phi,\psi)$ of a $3$-Lie algebra morphism $\phi:\mathfrak{g}\rightarrow \mathfrak{g}'$ and a linear
map $\psi:V\rightarrow V'$ satisfying
\begin{align}
&\psi(\rho(x,y)u)=\rho'(\phi(x),\phi(y))\psi(u),\;\forall x,y\in g,\;u\in V,\\
&\psi \circ \Theta=\Theta' \circ(\phi\otimes \phi \otimes \phi),\\
&\phi \circ T=T' \circ \psi.
\end{align}
\end{defi}
\begin{ex}
Any $\mathcal{O}$-operator (in particular, Rota-Baxter operator of weight $0$) on a $3$-Lie algebra is a $\Theta$-twisted $\mathcal{O}$-operator with $\Theta=0$.
\end{ex}
\begin{ex}
Let $\mathfrak{g}$ be a $3$-Lie algebra and $V$ be a $\mathfrak{g}$-module. Suppose $\theta:\mathfrak{g}\rightarrow V$ is an invertible $1$-cochain in the Chevalley-Eilenberg cochain complex
of $\mathfrak{g}$ with coefficients in $V$. Then $T=\theta^{-1}:V\rightarrow \mathfrak{g}$ is a $\Theta$-twisted $\mathcal{O}$-operator with $\Theta=-\partial \theta$. To verify this, we observe that
\begin{align}\label{eq;ex2}
&\Theta(Tu,Tv,Tw)=-(\partial \theta)(Tu,Tv,Tw)\nonumber\\
&=-\rho(Tu,Tv)\theta(Tw)-\rho(Tv,Tw)\theta(Tu)-\rho(Tw,Tu)\theta(Tv)
+\theta([Tu,Tv,Tw]_\mathfrak{g}).
\end{align}
By applying $T$ to the both sides of \eqref{eq;ex2}, we get the identity \eqref{eq:twisted}.
\end{ex}
\begin{ex}
Let $N : \mathfrak{g} \rightarrow \mathfrak{g}$ be a Nijenhuis operator on a $3$-Lie algebra $\mathfrak{g}$, i.e. $N$ satisfies
\begin{align}\label{eq:nijenhuis}
[Nx,Ny,Nz]_\mathfrak{g}&=N\Big([Nx,Ny,z]_\mathfrak{g}+[Nx,y,Nz]_\mathfrak{g}+[x,Ny,Nz]_\mathfrak{g}\nonumber\\
&-N([Nx,y,z]_\mathfrak{g}+[x,Ny,z]_\mathfrak{g}+[x,y,Nz]_\mathfrak{g})+N^{2}[x,y,z]_\mathfrak{g}\Big),\;for\;x,y,z\in \mathfrak{g}.
\end{align}
In this case, $\mathfrak{g}$ carries a new $3$-Lie algebra structure
\begin{align*}
[x, y,z]_N =&[Nx,Ny,z]_\mathfrak{g}+[Nx,y,Nz]_\mathfrak{g}+[x,Ny,Nz]_\mathfrak{g}-N([Nx,y,z]_\mathfrak{g}\\
&+[x,Ny,z]_\mathfrak{g}+[x,y,Nz]_\mathfrak{g}-N[x,y,z]_\mathfrak{g}).
\end{align*}
We denote this
$3$-Lie algebra  by $\mathfrak{g}_N$. Moreover, the $3$-Lie algebra $\mathfrak{g}_N$ has a representation on $\mathfrak{g}$ given by $\rho(x, y)z= [Nx, Ny,z]_\mathfrak{g}$,
for $x,y, z\in \mathfrak{g}$. With this representation, the map $\Theta:\wedge^{3}\mathfrak{g}_N\rightarrow \mathfrak{g}$ defined by $$\Theta(x,y,z) =-N([Nx,y,z]_\mathfrak{g}+[x,Ny,z]_\mathfrak{g}+[x,y,Nz]_\mathfrak{g}-N[x,y,z]_\mathfrak{g})$$ is a $2$-cocycle
in the Chevalley-Eilenberg cohomology of $\mathfrak{g}_N$ with coefficients in $\mathfrak{g}$. Then it is easy to observe that the
identity map $id : \mathfrak{g}\rightarrow \mathfrak{g}_N$ is a $\Theta$-twisted $\mathcal{O}$-operator.

This example will be more clear in Section 5 when we will introduce $3$-NS-Lie algebras and a functor
from the category of $3$-NS-Lie algebras to the category of twisted $\mathcal{O}$-operators.
\end{ex}
 Given a $\Theta$-twisted $\mathcal{O}$-operator $T$ and a $1$-cochain $\theta$,  we construct a $(\Theta +\partial \theta)$-twisted $\mathcal{O}$-operator under certain condition. First we observe the following.
\begin{pro}\label{iso}
Let $\mathfrak{g}$ be a $3$-Lie algebra and $V$ be a $\mathfrak{g}$-module. For any $2$-cocycle $\Theta \in \mathfrak{C}_{3Lie}^{2}(\mathfrak{g};V)$ and $1$-cochain $\theta \in \mathfrak{C}_{3Lie}^{1}(\mathfrak{g};V)$, we have an isomorphism of $3$-Lie algebras
$$\mathfrak{g}\ltimes_{\Theta}V\cong \mathfrak{g}\ltimes_{\Theta+\partial \theta}V.$$
\end{pro}
\begin{proof} Define $\psi_{\theta}:\mathfrak{g}\ltimes_{\Theta}V\rightarrow  \mathfrak{g}\ltimes_{\Theta+\partial \theta}V$ by $\psi_{\theta}(x,u)=(x,u-\theta(x))$, for $(x,u)\in \mathfrak{g}\oplus V$.
Then we have
\begin{align*}
&\psi_{\theta}([(x,u),(y,v),(z,w)]_{\Theta})\\
&=\Big([x,y,z]_\mathfrak{g},\rho(x,y)w+\rho(z,x)v+\rho(y,z)u+\Theta(x,y,z)-\theta([x,y,z]_\mathfrak{g})\Big)\\
&=\Big([x,y,z]_\mathfrak{g},\rho(x,y)w+\rho(z,x)v+\rho(y,z)u+\Theta(x,y,z)\\
&-\rho(x,y)\theta(z)-\rho(z,x)\theta(y)-\rho(y,z)\theta(x)+(\partial \theta)(x,y,z)\Big)\\
&=[(x,u-\theta(x)),(y,v-\theta(y)),(z,w-\theta(z))]_{\Theta+\partial \theta}.
\end{align*}
This proves the result.
\end{proof}
\begin{pro}
Let $T:V\rightarrow \mathfrak{g}$ be a $\Theta$-twisted $\mathcal{O}$-operator. For any $1$-cochain $\theta\in \mathfrak{C}_{3Lie}^{1}(\mathfrak{g};V)$, if the linear map
$(Id_V-\theta \circ T):V\rightarrow V$ is invertible, then the linear map $T\circ (Id_V-\theta \circ T)^{-1}:V\rightarrow \mathfrak{g}$ is a $(\Theta+\partial \theta)$-twisted
$\mathcal{O}$-operator.
\end{pro}
\begin{proof} Consider the subalgebra $Gr(T)\subset \mathfrak{g}\ltimes_{\Theta}V$ of the $\Theta$-twisted semi-direct product. Thus by\\ Proposition \ref{iso}, we get that
$$\psi_{\theta}(Gr(T))=\{(Tu,u-(\theta\circ T)(u)|\;u\in V\}\subset \mathfrak{g}\ltimes_{\Theta+\partial \theta}V$$
is a subalgebra. Since the map $(Id_V-\theta \circ T):V\rightarrow V$ is invertible, we have $\psi_{\theta}(Gr(T))$ is the graph of the linear map
$T\circ (Id_V-\theta \circ T)^{-1}$. In this case, it follows from Proposition \ref{graph} that $T\circ (Id_V-\theta \circ T)^{-1}$ is a $(\Theta+\partial \theta)$-twisted $\mathcal{O}$-operator.
\end{proof}
Next, we give a construction of a new $\Theta$-twisted $\mathcal{O}$-operator out of an old one and a
suitable $1$-cocycle. Let $T:V\rightarrow \mathfrak{g}$ be a $\Theta$-twisted $\mathcal{O}$-operator. Suppose $\theta\in \mathfrak{C}_{3Lie}^{1}(\mathfrak{g};V)$ is a
$1$-cocycle in the Chevalley-Eilenberg cochain complex of $\mathfrak{g}$ with coefficients in $V$. Then $\theta$ is
said to be $T$-admissible if the linear map $(Id_V + \theta\circ T) : V \rightarrow V$ is invertible. 
\begin{pro}
Let $\theta\in \mathfrak{C}_{3Lie}^{1}(\mathfrak{g};V)$ be a $T$-admissible $1$-cocycle. Then $T\circ (Id_V + \theta\circ T)^{-1}:V\rightarrow \mathfrak{g}$ is a $\Theta$-twisted $\mathcal{O}$-operator.
\end{pro}
\begin{proof} Consider the deformed subspace
$$\tau_{\theta}(Gr(T))=\{(Tu,u+(\theta\circ T)(u))|\;u\in V\}\subset \mathfrak{g}\ltimes_{\Theta}V.$$
Since $\theta$ is a $1$-cocycle, $\tau_{\theta}(Gr(T))\subset g\ltimes_{\Theta}V$ turns out to be a subalgebra. Further,
the map $(Id_V + \theta \circ T)$ is invertible implies that $\tau_{\theta}(Gr(T))$ is the graph of the map
$T\circ (Id_V + \theta\circ T)^{-1}$. Hence the result follows from Proposition \ref{graph}.
\end{proof}

The $\Theta$-twisted $\mathcal{O}$-operator in the above proposition is called the gauge transformation
of $T$ associated with $\theta$. We denote this $\Theta$-twisted $\mathcal{O}$-operator
simply by $T_\theta$.
\begin{pro}
Let $T$ be a $\Theta$-twisted $\mathcal{O}$-operator and $\theta$ be a $T$-admissible $1$-cocycle. Then the $3$-Lie algebra structures on $V$ induced from the $\Theta$-twisted $\mathcal{O}$-operators $T$ and $T_\theta$ are isomorphic.
\end{pro}
\begin{proof}
 Consider the linear isomorphism $(Id_V + \theta\circ T) : V\rightarrow V$. Moreover, for any
$u, v,w\in V$, we have
\begin{align*}
&[(Id_V +\theta\circ T)(u),(Id_V + \theta\circ T)(v),(Id_V + \theta\circ T)(w)]_{T_B}\\
&=\rho(Tv,Tw)(Id_V + \theta\circ T)(u)+\rho(Tw,Tu)(Id_V + \theta\circ T)(v)\\
&+\rho(Tu,Tv)(Id_V + \theta\circ T)+\Theta(Tu,Tv,Tw)\\
&=\rho(Tv,Tw)u+\rho(Tw,Tu)v+\rho(Tu,Tv)w+\Theta(Tu,Tv,Tw)\\
&+\rho(Tv,Tw)(\theta\circ T)(u)+\rho(Tw,Tu)(\theta\circ T)(v)+\rho(Tu,Tv)(\theta\circ T)(w)\\
&=[u,v,w]_T+\theta([Tu,Tv,Tw]_\mathfrak{g})\\
&=[u,v,w]_T+\theta\circ T([u,v,w]_T)=(Id_V + \theta\circ T)([u,v,w]_T).
\end{align*}
This shows that $(Id_V +\theta\circ T) : (V, [\cdot,\cdot,\cdot]_T) \rightarrow (V, [\cdot,\cdot,\cdot]_{T_\theta} )$ is a $3$-Lie algebra isomorphism.
\end{proof}

\section{Cohomology of twisted $\mathcal{O}$-operators}
%%%%%%%%%%%%%%%%%%%%%
In this section, we define a cohomology of a $\Theta$-twisted $\mathcal{O}$-operator
 as the Chevalley-Eilenberg cohomology of the $3$-Lie algebra $(V, [\cdot,\cdot, \cdot]_T)$ constructed in
Proposition \ref{eq:cro1} with coefficients in a suitable representation on $\mathfrak{g}$. This cohomology will be used in  Section 4  to study deformations of $T$.
\begin{pro}
Let $T$ be a $\Theta$-twisted $\mathcal{O}$-operator on a $3$-Lie algebra $(\mathfrak{g}, [\cdot,\cdot, \cdot]_\mathfrak{g})$ with respect
to a representation $(V,\rho)$. Define $\rho_{\Theta}:\wedge^{2}V\rightarrow gl(\mathfrak{g})$ by
\begin{equation}
\rho_{\Theta}(u,v)x=[Tu,Tv,x]_\mathfrak{g}-T\Big(\rho(Tv,x)u+\rho(x,Tu)v+\Theta(x,Tu,Tv)\Big),\;\forall u,v\in V,x\in \mathfrak{g}.
\end{equation}
Then $(\mathfrak{g},\rho_{\Theta} )$ is a representation of the $3$-Lie algebra $( V, [\cdot,\cdot, \cdot]_T)$ on the vector space $\mathfrak{g}$.
\end{pro}
\begin{proof} By a direct calculation using the definition of $\rho_{\Theta}$, we get
\begin{align*}
&\rho_{\Theta}(u_1,u_2)\rho_{\Theta}(u_3,u_4)x-\rho_{\Theta}([u_1,u_2,u_3]_T,u_4)x\\
&-\rho_{\Theta}(u_3,[u_1,u_2,u_4]_T)x-\rho_{\Theta}(u_3,u_4)\rho_{T}(u_1,u_2)x\\
&=\quad[Tu_1,Tu_2,[Tu_3,Tu_4,x]_\mathfrak{g}]_\mathfrak{g}+[Tu_1,Tu_2,T\rho(Tu_3,x)u_4]_\mathfrak{g}-[Tu_1,Tu_2,T\rho(Tu_4,x)u_3]_\mathfrak{g}\\
&-[Tu_1,Tu_2,T\Theta(x,Tu_3,Tu_4)]_\mathfrak{g}+T\rho(Tu_1,[Tu_3,Tu_4,x]_\mathfrak{g})u_2+T\rho(Tu_1,T\rho(Tu_3,x)u_4)u_2\\
&-T\rho(Tu_1,T\rho(Tu_4,x)u_3)u_2-T\rho(Tu_1,T\Theta(x,Tu_3,Tu_4))u_2-T\rho(Tu_2,[Tu_3,Tu_4,x]_\mathfrak{g})u_1\\
&-T\rho(Tu_2,T\rho(Tu_3,x)u_4)u_1+T\rho(Tu_2,T\rho(Tu_4,x)u_3)u_1+T\rho(Tu_2,T\Theta(x,Tu_3,Tu_4))u_1\\
&-T\Theta(T\rho(Tu_3,x)u_4,Tu_1,Tu_2)+T\Theta(T\rho(Tu_4,x)u_3,Tu_1,Tu_2)\\
&+T\Theta(T\Theta(x,Tu_3,T_4),Tu_1,Tu_2)-T\Theta([Tu_3,Tu_4,x]_\mathfrak{g},Tu_1,Tu_2)\\
&-[[Tu_1,Tu_2,Tu_3]_\mathfrak{g},Tu_4,x]_\mathfrak{g}-T(\rho([Tu_1,Tu_2,Tu_3]_\mathfrak{g},x)u_4+T\rho(Tu_4,x)\rho(Tu_1,Tu_2)u_3\\
&+T\rho(Tu_4,x)\rho(Tu_2,Tu_3)u_1+T\rho(Tu_4,x)\rho(Tu_3,Tu_1)u_2+T\rho(Tu_4,x)\Theta(Tu_1,Tu_2,Tu_3)\\
&+T\Theta(x,[Tu_1,Tu_2,Tu_3]_\mathfrak{g},Tu_4)-[Tu_3,[Tu_1,Tu_2,Tu_4]_\mathfrak{g},x]_\mathfrak{g}-T\rho(Tu_3,x)\rho(Tu_1,Tu_2)u_4\\
&-T\rho(Tu_3,x)\rho(Tu_2,Tu_4)u_1-T\rho(Tu_3,x)\rho(Tu_4,Tu_1)u_2-T\rho(Tu_3,x)\Theta(Tu_1,Tu_2,Tu_4)\\
&+T(\rho([Tu_1,Tu_2,Tu_4]_\mathfrak{g},x)u_3+T\Theta(x,Tu_3,[Tu_1,Tu_2,Tu_4]_\mathfrak{g})\\
&-[Tu_3,Tu_4,[Tu_1,Tu_2,x]_\mathfrak{g}]_\mathfrak{g}-[Tu_3,Tu_4,T\rho(Tu_1,x)u_2]_\mathfrak{g}+[Tu_3,Tu_4,T\rho(Tu_2,x)u_1]_\mathfrak{g}\\
&+[Tu_3,Tu_4,T\Theta(x,Tu_1,Tu_2)]_\mathfrak{g}-T\rho(Tu_3,[Tu_1,Tu_2,x]_\mathfrak{g})u_4-T\rho(Tu_3,T\rho(Tu_1,x)u_2)u_4\\
&+T\rho(Tu_3,T\rho(Tu_2,x)u_1)u_4+T\rho(Tu_3,T\Theta(x,Tu_1,Tu_2))u_4+T\rho(Tu_4,[Tu_1,Tu_2,x]_\mathfrak{g})u_3\\
&+T\rho(Tu_4,T\rho(Tu_1,x)u_2)u_3-T\rho(Tu_4,T\rho(Tu_2,x)u_1)u_3-T\rho(Tu_4,T\Theta(x,Tu_1,Tu_2))u_3\\
&+T\Theta(T\rho(Tu_1,x)u_2,Tu_3,Tu_4)-T\Theta(T\rho(Tu_2,x)u_1,Tu_3,Tu_4)\\
&-T\Theta(T\Theta(x,Tu_1,Tu_2),Tu_3,Tu_4)+T\Theta([Tu_1,Tu_2,x]_\mathfrak{g},Tu_3,Tu_4)\\
^{\eqref{eq:de1}+\eqref{eq:rep}+\eqref{eq:twisted}}&=-T\big(\Theta(Tu_1,Tu_2,[Tu_3,Tu_4,x]_\mathfrak{g})+\rho(Tu_1,Tu_2)\Theta(Tu_3,Tu_4,x)\\
&-\rho(Tu_4,x)\Theta(Tu_1,Tu_2,Tu_3)-\Theta([Tu_1,Tu_2,Tu_3]_\mathfrak{g},Tu_4,x)\\
&-\rho(x,Tu_3)\Theta(Tu_1,Tu_2,Tu_4)-\Theta(Tu_3,[Tu_1,Tu_2,Tu_4]_\mathfrak{g},x)\\
&-\rho(Tu_3,Tu_4)\Theta(Tu_1,Tu_2,x)-\Theta(Tu_3,Tu_4,[Tu_1,Tu_2,x]_\mathfrak{g})\big)\\
^{\eqref{eq:cocycle}}&=-T\big((\partial \Theta)(Tu_1,Tu_2,Tu_3,Tu_4,x)\big)=0.
\end{align*}
Similarly,
\begin{align*}
&\rho_{\Theta}([u_1,u_2,u_3]_T,u_4)x-\rho_{\Theta}(u_1,u_2)\rho_{\Theta}(u_3,u_4)x\\
&-\rho_{\Theta}(u_2,u_3)\rho_{\Theta}(u_1,u_4)x-\rho_{\Theta}(u_3,u_1)\rho_{\Theta}(u_2,u_4)x=0.
\end{align*}
Hence the result follows.\end{proof}

Let $\partial_{\Theta}: \mathfrak{C}_{3Lie}^{n}(V;\mathfrak{g})\rightarrow \mathfrak{C}_{3Lie}^{n+1}(V;\mathfrak{g})$,$\;(n \geq 1)$ be the corresponding coboundary operator
of the $3$-Lie algebra $(V, [\cdot,\cdot,\cdot]_T )$ with coefficients in the representation $(\mathfrak{g},\rho_\Theta )$. More precisely,\\
$\partial_{\Theta}: \mathfrak{C}_{3Lie}^{n}(V;\mathfrak{g})\rightarrow \mathfrak{C}_{3Lie}^{n+1}(V;\mathfrak{g})$ is given by
\begin{align}
&\quad\;(\partial_{\Theta} f)(u_1,\cdots,u_{2n+1})\nonumber\\
&=(-1)^{n+1}\Big([Tu_{2n+1},Tu_{2n-1},f(u_1,\cdots,u_{2n-2},u_{2n})]_\mathfrak{g}-T\rho(Tu_{2n-1},f(u_1,\cdots,u_{2n-2},u_{2n}))u_{2n+1}\nonumber\\
&-T\rho(f(u_1,\cdots,u_{2n-2},u_{2n}),Tu_{2n+1})u_{2n-1}
-T\Theta(f(u_1,\cdots,u_{2n-2},u_{2n}),Tu_{2n+1},Tu_{2n-1})\Big)\nonumber\\
&+(-1)^{n+1}\Big([Tu_{2n},Tu_{2n+1},f(u_1,\cdots,u_{2n-1})]_\mathfrak{g}-T\rho(Tu_{2n+1},f(u_1,\cdots,u_{2n-1}))u_{2n}\nonumber\\
&-T\rho(f(u_1,\cdots,u_{2n-1}),Tu_{2n})u_{2n+1}
-T\Theta(f(u_1,\cdots,u_{2n-1}),Tu_{2n},Tu_{2n+1})\Big)\nonumber\\
&+\sum_{k=1}^{n}(-1)^{k+1}\Big([Tu_{2k-1},Tu_{2k},f(u_1,u_2,\cdots,\widehat{u_{2k-1}},\widehat{u_{2k}},\cdots,u_{2n+1})]_\mathfrak{g}\nonumber\\
&-T\rho(Tu_{2k},f(u_1,u_2,\cdots,\widehat{u_{2k-1}},\widehat{u_{2k}},\cdots,u_{2n+1}))u_{2k-1}\nonumber\\
&-T\rho(f(u_1,u_2,\cdots,\widehat{u_{2k-1}},\widehat{u_{2k}},\cdots,u_{2n+1}),Tu_{2k-1})u_{2k}\nonumber\\
&-T\Theta(f(u_1,u_2,\cdots,\widehat{u_{2k-1}},\widehat{u_{2k}},\cdots,u_{2n+1}),Tu_{2k-1},Tu_{2k})\Big)\nonumber\\
&+\sum_{k=1}^{n}\sum_{j=2k+1}^{2n+1}(-1)^{k}f\Big(u_1,u_2,\cdots,\widehat{u_{2k-1}},\widehat{u_{2k}},\cdots,\rho(Tu_{2k-1},Tu_{2k})u_j\nonumber\\
&+\rho(Tu_{2k},Tu_j)u_{2k-1}+\rho(Tu_j,Tu_{2k-1})u_{2k}+\Theta(Tu_{2k-1},Tu_{2k},Tu_j),\cdots,u_{2n+1}\Big),
\end{align}
for $f\in \mathfrak{C}_{3Lie}^{n}(V;\mathfrak{g})$ and $u_1,u_2,...,u_{2n+1}\in V$.\\

 It is obvious that $f\in \mathfrak{C}_{3Lie}^{1}(V;\mathfrak{g})$ is closed if and only if
\begin{align*}
&[Tu,Tv,f(w)]_\mathfrak{g}+[f(u),Tv,Tw]_\mathfrak{g}+[Tu,f(v),Tw]_\mathfrak{g}\\
&-f\Big(\rho(Tu,Tv)w+\rho(Tv,Tw)u+\rho(Tw,Tu)v+\Theta(Tu,Tv,Tw)\Big)\\
&-T\Big(\rho(Tv,f(w))u+\rho(f(w),Tu)v+\Theta(f(w),Tu,Tv)\Big)\\
&-T\Big(\rho(Tw,f(u))v+\rho(f(u),Tv)w+\Theta(f(u),Tv,Tw)\Big)\\
&-T\Big(\rho(Tu,f(v))w+\rho(f(v),Tw)u+\Theta(f(v),Tw,Tu)\Big)=0.
\end{align*}
For any $\mathfrak{X}\in \mathfrak{g}\wedge \mathfrak{g}$, we define $\delta(\mathfrak{X}):V\rightarrow \mathfrak{g}$ by
\begin{equation}
\delta(\mathfrak{X})v=T\big(\rho(\mathfrak{X})v+\Theta(\mathfrak{X},Tv)\big)-[\mathfrak{X},Tv]_\mathfrak{g},\;\forall v\in V.
\end{equation}
\begin{pro}
Let $T$ be a $\Theta$-twisted $\mathcal{O}$-operator on a $3$-Lie algebra $(\mathfrak{g}, [\cdot,\cdot, \cdot]_\mathfrak{g})$ with
respect to a representation $(V,\rho)$. Then $\delta(\mathfrak{X})$ is a $1$-cocycle on the $3$-Lie algebra $(V, [\cdot,\cdot,\cdot]_T )$
with coefficients in $(\mathfrak{g},\rho_\Theta )$.
\end{pro}
\begin{proof} For any $u,v,w\in V$, we have
\begin{align*}
&\quad\;(\partial_{\Theta}\delta(\mathfrak{X}))(u,v,w)\\
&=[Tu,Tv,\delta(\mathfrak{X})(w)]_\mathfrak{g}+[\delta(\mathfrak{X})(u),Tv,Tw]_\mathfrak{g}+[Tu,\delta(\mathfrak{X})(v),Tw]_\mathfrak{g}\\
&-\delta(\mathfrak{X})\big(\rho(Tu,Tv)w+\rho(Tv,Tw)u+\rho(Tw,Tu)v+\Theta(Tu,Tv,Tw)\big)\\
&-T\big(\rho(Tv,\delta(\mathfrak{X})(w))u+\rho(\delta(\mathfrak{X})(w),Tu)v+\Theta(\delta(\mathfrak{X})(w),Tu,Tv)\big)\\
&-T\big(\rho(Tw,\delta(\mathfrak{X})(u))v+\rho(\delta(\mathfrak{X})(u),Tv)w+\Theta(\delta(\mathfrak{X})(u),Tv,Tw)\big)\\
&-T\big(\rho(Tu,\delta(\mathfrak{X})(v))w+\rho(\delta(\mathfrak{X})(v),Tw)u+\Theta(\delta(\mathfrak{X})(v),Tw,Tu)\big)\\
&=[Tu,Tv,T\rho(\mathfrak{X})w]_\mathfrak{g}+[Tu,Tv,T\Theta(\mathfrak{X},w)]_\mathfrak{g}-[Tu,Tv,[\mathfrak{X},Tw]_\mathfrak{g}]_\mathfrak{g}\\
&+[T\rho(\mathfrak{X})u,Tv,Tw]_\mathfrak{g}+[T\Theta(\mathfrak{X},Tu),Tv,Tw]_\mathfrak{g}-[[\mathfrak{X},Tu]_\mathfrak{g},Tv,Tw]_\mathfrak{g}\\
&+[Tu,T\rho(\mathfrak{X})v,Tw]_\mathfrak{g}+[Tu,T\Theta(\mathfrak{X},Tv),Tw]_\mathfrak{g}-[Tu,[\mathfrak{X},Tv]_\mathfrak{g},Tw]_\mathfrak{g}\\
&-T\rho(\mathfrak{X})\big(\rho(Tu,Tv)w+\rho(Tv,Tw)u+\rho(Tw,Tu)v+\Theta(Tu,Tv,Tw)\big)\\
&-T\Theta(\mathfrak{X},[Tu,Tv,Tw]_\mathfrak{g})+[\mathfrak{X},[Tu,Tv,Tw]_\mathfrak{g}]_\mathfrak{g}\\
&-T\big(\rho(Tv,T\rho(\mathfrak{X})w)u+\rho(Tv,T\Theta(\mathfrak{X},Tw))u-\rho(Tv,[\mathfrak{X},Tw]_\mathfrak{g})u\\
&+\rho(T\rho(\mathfrak{X})w,Tu)v+\rho(T\Theta(\mathfrak{X},Tw),Tu)v-\rho([\mathfrak{X},Tw]_\mathfrak{g},Tu)v\\
&+\Theta(T\rho(\mathfrak{X})w,Tu,Tv)+\Theta(T\Theta(\mathfrak{X},Tw),Tu,Tv)-\Theta([\mathfrak{X},Tw]_\mathfrak{g},Tu,Tv)\big)\\
&-T\big(\rho(Tw,T\rho(\mathfrak{X})u)v+\rho(Tw,T\Theta(\mathfrak{X},Tu))v-\rho(Tw,[\mathfrak{X},Tu]_\mathfrak{g})v\\
&+\rho(T\rho(\mathfrak{X})u,Tv)w+\rho(T\Theta(\mathfrak{X},Tu),Tv)w-\rho([\mathfrak{X},Tu]_\mathfrak{g},Tv)w\\
&+\Theta(T\rho(\mathfrak{X})u,Tv,Tw)+\Theta(T\Theta(\mathfrak{X},Tu),Tv,Tw)-\Theta([\mathfrak{X},Tu]_\mathfrak{g},Tv,Tw)\big)\\
&-T\big(\rho(Tu,T\rho(\mathfrak{X})v)w+\rho(Tu,T\Theta(\mathfrak{X},Tv))w-\rho(Tu,[\mathfrak{X},Tv]_\mathfrak{g})w\\
&+\rho(T\rho(\mathfrak{X})v,Tw)u+\rho(T\Theta(\mathfrak{X},Tv),Tw)u-\rho([\mathfrak{X},Tv]_\mathfrak{g},Tw)u\\
&+\Theta(T\rho(\mathfrak{X})v,Tw,Tu)+\Theta(T\Theta(\mathfrak{X},Tv),Tw,Tu)-\Theta([\mathfrak{X},Tv]_\mathfrak{g},Tw,Tu)\big)\\
^{\eqref{eq:twisted}+\eqref{eq:de1}+\eqref{eq:rep}}&=-T\big(\Theta(\mathfrak{X},[Tu,Tv,Tw]_\mathfrak{g})+T\rho(\mathfrak{X})\Theta(Tu,Tv,Tw)\\
&-T\rho(Tu,Tv)\Theta(\mathfrak{X},Tw)-T\rho(Tv,Tw)\Theta(\mathfrak{X},Tu)-T\rho(Tw,Tu)\Theta(\mathfrak{X},Tv)\\
&-T\Theta([\mathfrak{X},Tw]_\mathfrak{g},Tu,Tv)-T\Theta([\mathfrak{X},Tu]_\mathfrak{g},Tv,Tw)-T\Theta([\mathfrak{X},Tv]_\mathfrak{g},Tw,Tu)\big)\\
&=-T\big((\partial \Theta)(\mathfrak{X},Tu,Tv,Tw)\big)=0.
\end{align*}
Thus, we deduce that $\partial_{\Theta}\delta(\mathfrak{X})=0$.\end{proof}

Now, we give a cohomology of $\Theta$-twisted $\mathcal{O}$-operators on $3$-Lie algebras.
\begin{defi}
Let $T$ be a $\Theta$-twisted $\mathcal{O}$-operator on a $3$-Lie algebra $(\mathfrak{g}, [\cdot,\cdot, \cdot]_\mathfrak{g})$ with
respect to a representation $(V,\rho)$. Define the set of $n$-cochains by
\begin{equation}
\mathfrak{C}_{\Theta}^{n}(V;\mathfrak{g})=
\begin{cases}
\mathfrak{C}_{3Lie}^{n}(V;\mathfrak{g}),\;n\geq 1,\\
\mathfrak{g}\wedge \mathfrak{g},\quad \quad \quad  n=0.
\end{cases}
\end{equation}
Define $D_\Theta:\mathfrak{C}_{\Theta}^{n}(V;\mathfrak{g})\rightarrow \mathfrak{C}_{\Theta}^{n+1}(V;\mathfrak{g})$ by
\begin{equation}
D_\Theta=
\begin{cases}
\partial_{\Theta},\;n\geq 1,\\
\delta,\;\;\;\; n=0.
\end{cases}
\end{equation}
\end{defi}
Denote the set of $n$-cocycles by $\mathcal{Z}_{\Theta}^{n}(V;\mathfrak{g})$ and the set of $n$-coboundaries by $\mathcal{B}_{\Theta}^{n}(V;\mathfrak{g})$. Denote by
$$\mathcal{H}_{\Theta}^{n}(V;\mathfrak{g})=\mathcal{Z}_{\Theta}^{n}(V;\mathfrak{g})/\mathcal{B}_{\Theta}^{n}(V;\mathfrak{g}),\;n\geq 0$$
the $n$-th cohomology group which will be taken to be the $n$-th cohomology group for the
$\Theta$-twisted $\mathcal{O}$-operator $T$.

\section{Deformations of twisted $\mathcal{O}$-operators}
%%%%%%%%%%%%%%%%%%%%%
In this section, we study infinitesimal and formal deformations of a $\Theta$-twisted $\mathcal{O}$-operator.
\subsection{ Infinitesimal deformations}
Let $(\mathfrak{g}, [\cdot,\cdot, \cdot]_\mathfrak{g})$ be a $3$-Lie algebra, $(V,\rho)$ be a representation of $\mathfrak{g}$, and $\Theta\in \mathfrak{C}_{3Lie}^{2}(\mathfrak{g};V)$
be a $2$-cocycle in the Chevalley-Eilenberg cochain complex. Let $T : V \rightarrow \mathfrak{g}$ be a $\Theta$-twisted $\mathcal{O}$-operator.
\begin{defi}
An infinitesimal  deformation of $T$ consists of a parametrized sum $T_t = T + tT_1$, for some $T_1 \in
Hom(V, \mathfrak{g})$ such that $T_t$ is a $\Theta$-twisted $\mathcal{O}$-operator for all values of $t$. In this case, we say that
$T_1$ generates an infinitesimal  deformation of $T$ .
\end{defi}
Suppose $T_1$ generates an infinitesimal  deformation of $T$. Then we have
\begin{align*}
&[T_t(u),T_t(v),T_t(w)]_\mathfrak{g}\\
&=T_t\Big(\rho(T_t(u),T_t(v))w+\rho(T_t(v),T_t(w))u+\rho(T_t(w),T_t(u))v+\Theta(T_t(u),T_t(v),T_t(w))\Big),
\end{align*}
for $u,v,w\in V$. This is equivalent to the following conditions
\begin{align}\label{eq:defo}
&[Tu,Tv,T_1(w)]_\mathfrak{g}+[Tu,T_1(v),Tw]_\mathfrak{g}+[T_1(u),Tv,Tw]_\mathfrak{g}\nonumber\\
&=T\Big(\rho(Tu,T_1(v))w+\rho(T_1(u),Tv)w+\rho(Tv,T_1(w))u+\rho(T_1(v),Tw)u\nonumber\\
&+\rho(Tw,T_1(u))v+\rho(T_1(w),Tu)v+\Theta(Tu,Tv,T_1(w))+\Theta(Tu,T_1(v),Tw)\nonumber\\
&+\Theta(T_1(u),Tv,Tw)\Big)\nonumber\\
&+T_1\Big(\rho(Tu,Tv)w+\rho(Tv,Tw)u+\rho(Tw,Tu)v+\Theta(Tu,Tv,Tw)\Big),
\end{align}
\begin{align}
&[Tu,T_1(v),T_1(w)]_\mathfrak{g}+[T_1(u),Tv,T_1(w)]_\mathfrak{g}+[T_1(u),T_1(v),Tw]_\mathfrak{g}\nonumber\\
&=T\Big(\rho(T_1(u),T_1(v))w+\rho(T_1(v),T_1(w))u+\rho(T_1(w),T_1(u))v\nonumber\\
&+\Theta(Tu,T_1(v),T_1(w))+\Theta(T_1(u),Tv,T_1(w))+\Theta(T_1(u),T_1(v),Tw)\Big)\nonumber\\
&+T_1\Big(\rho(Tu,T_1(v))w+\rho(T_1(u),Tv)w+\rho(Tv,T_1(w))u+\rho(T_1(v),Tw)u\nonumber\\
&+\rho(Tw,T_1(u))v+\rho(T_1(w),Tu)v+\Theta(Tu,Tv,T_1(w))+\Theta(Tu,T_1(v),Tw)\nonumber\\
&+\Theta(T_1(u),Tv,Tw)\Big),
\end{align}
\begin{align}
&[T_1(u),T_1(v),T_1(w)]_\mathfrak{g}=T\Big(\Theta(T_1(u),T_1(v),T_1(w))\Big)\nonumber\\
&+T\Big(\rho(T_1(u),T_1(v))w+\rho(T_1(v),T_1(w))u+\rho(T_1(w),T_1(u))v\nonumber\\
&+\Theta(Tu,T_1(v),T_1(w))+\Theta(T_1(u),Tv,T_1(w))+\Theta(T_1(u),T_1(v),Tw)\Big)
\end{align}
and
\begin{align}
&T_1\Big(\Theta(T_1(u),T_1(v),T_1(w))\Big)=0.
\end{align}
Note that the identity \eqref{eq:defo} implies that $T_1$ is a $1$-cocycle in the cohomology of $T$. Hence, $T_1$ defines a
cohomology class in $\mathcal{H}_{\Theta}^{1}(V;\mathfrak{g})$.
\begin{defi}
Two infinitesimal  deformations $T_t = T + tT_1$ and $T^{'}_t = T + tT^{'}_1$ of a $\Theta$-twisted $\mathcal{O}$-operator $T$ are said to be equivalent if there exists an element $\mathfrak{X}\in \mathfrak{g}\wedge \mathfrak{g}$ such that the pair
\begin{equation}
\Big(\phi_{t}=Id_\mathfrak{g}+t[\mathfrak{X},-]_\mathfrak{g},\;\psi_{t}=Id_V+t\big(\rho(\mathfrak{X})(-)+\Theta(\mathfrak{X},T-)\big)\Big),
\end{equation}
defines a morphism of $\Theta$-twisted $\mathcal{O}$-operators from $T_t$ to $T^{'}_t$.

An infinitesimal deformation $T_t = T + tT_1$ of a $\Theta$-twisted $\mathcal{O}$-operator is said to be
trivial if $T_t$ is equivalent to  $T^{'}_t = T$.
\end{defi}
The condition that $\phi_{t}=Id_\mathfrak{g}+t[\mathfrak{X},-]_\mathfrak{g}$ is a $3$-Lie algebra morphism of $(\mathfrak{g}, [\cdot,\cdot, \cdot]_\mathfrak{g})$ is equivalent
to
\begin{equation}
\begin{cases}
[z_1,[\mathfrak{X},z_2]_\mathfrak{g},[\mathfrak{X},z_3]_\mathfrak{g}]_\mathfrak{g}+[[\mathfrak{X},z_1]_\mathfrak{g},z_2,[\mathfrak{X},z_3]_\mathfrak{g}]_\mathfrak{g}\\
+[[\mathfrak{X},z_1]_\mathfrak{g},[\mathfrak{X},z_2]_\mathfrak{g},z_3]_\mathfrak{g}=0,\\
[[\mathfrak{X},z_1]_\mathfrak{g},[\mathfrak{X},z_2]_\mathfrak{g},[\mathfrak{X},z_3]_\mathfrak{g}]_\mathfrak{g}=0,\;for\;z_1,z_2,z_3\in \mathfrak{g}.
\end{cases}
\end{equation}
The condition $\psi_t(\rho(z_1,z_2)u)=\rho(\phi_t(z_1),\phi_(z_2))\psi_t(u)$ implies that
\begin{equation}
\begin{cases}
\Theta(\mathfrak{X},T\rho(z_1,z_2)u)=\rho(z_1,z_2)\Theta(\mathfrak{X},Tu),\\
\big(\rho(z_1,[\mathfrak{X},z_2]_\mathfrak{g})+\rho([\mathfrak{X},z_1]_\mathfrak{g},z_2)\big)\big(\rho(\mathfrak{X})u+\Theta(\mathfrak{X},Tu)\big)\\
+\rho([\mathfrak{X},z_1]_\mathfrak{g},[\mathfrak{X},z_2]_\mathfrak{g})u=0,\\
\rho([\mathfrak{X},z_1]_\mathfrak{g},[\mathfrak{X},z_2]_\mathfrak{g})\big(\rho(\mathfrak{X})u+\Theta(\mathfrak{X},Tu)\big)=0,
\end{cases}
\end{equation}
Finally, the conditions $\psi_t\circ \Theta=\Theta \circ (\phi_t\otimes \phi_t \otimes \phi_t)$ and $\phi_t\circ T_t=T^{'}_t\circ \psi_t$
are respectively equivalent to
\begin{equation}
\begin{cases}
\rho(\mathfrak{X})\Theta(z_1,z_2,z_3)+\Theta(\mathfrak{X},T\Theta(z_1,z_2,z_3))=\Theta([\mathfrak{X},z_1]_\mathfrak{g},z_2,z_3)\\
+\Theta(z_1,[\mathfrak{X},z_2]_\mathfrak{g},z_3)+\Theta(z_1,z_2,[\mathfrak{X},z_3]_\mathfrak{g}),\\
\Theta(z_1,[\mathfrak{X},z_2]_\mathfrak{g},[\mathfrak{X},z_3]_\mathfrak{g})+\Theta([\mathfrak{X},z_1]_\mathfrak{g},z_2,[\mathfrak{X},z_3]_\mathfrak{g})\\
+\Theta([\mathfrak{X},z_1]_\mathfrak{g},[\mathfrak{X},z_2]_\mathfrak{g},z_3)=0,\\
\Theta([\mathfrak{X},z_1]_\mathfrak{g},[\mathfrak{X},z_2]_\mathfrak{g},[\mathfrak{X},z_3]_\mathfrak{g})=0,
\end{cases}
\end{equation}
\begin{equation}\label{eq:defoo}
\begin{cases}
T_1(u)+[\mathfrak{X},Tu]_\mathfrak{g}=T\big(\rho(\mathfrak{X})u+\Theta(\mathfrak{X},Tu)\big)+T^{'}_1(u),\\
[\mathfrak{X},T_1(u)]_\mathfrak{g}=T^{'}_1\big(\rho(\mathfrak{X})u+\Theta(\mathfrak{X},Tu)\big).
\end{cases}
\end{equation}
Note that the above identities hold for all $\mathfrak{X}\in \mathfrak{g}\wedge \mathfrak{g}$, $z_1,z_2,z_3\in \mathfrak{g}$ and $u\in V$.\\

From the first condition of \eqref{eq:defoo}, we have
\begin{align*}
&T_1(u)-T^{'}_1(u)\\
&=T\big(\rho(\mathfrak{X})u+\Theta(\mathfrak{X},Tu)\big)-[\mathfrak{X},Tu]_\mathfrak{g}=D_{\Theta}(\mathfrak{X})(u).
\end{align*}
Therefore, we get the following theorem.
\begin{thm}
Let $T_t = T + tT_1$ and $T^{'}_t = T + tT^{'}_1$ be two equivalent infinitesimal deformations of a $\Theta$-twisted
$\mathcal{O}$-operator. Then $T_1$ and $T^{'}_1$ defines the same cohomology class in $\mathcal{H}_{\Theta}^{1}(V;\mathfrak{g})$.
\end{thm}

\subsection{ Formal deformations}
Let $\mathfrak{g}$ be a $3$-Lie algebra, $V$ be a $\mathfrak{g}$-module and $\Theta$ a $2$-cocycle in the Chevalley-Eilenberg cohomology of $\mathfrak{g}$ with coefficients in $V$.
Let $T:V\rightarrow \mathfrak{g}$ be a $\Theta$-twisted $\mathcal{O}$-operator.\\

Let $\mathbb{K}[[t]]$ be the ring of power series in one variable $t$. For any $\mathbb{K}$-linear space $V$, we let $V[[t]]$ denote the
vector space of formal power series in $t$ with coefficients in $V$. If in addition, $(\mathfrak{g}, [\cdot,\cdot, \cdot]_\mathfrak{g})$ is a
$3$-Lie algebra over $\mathbb{K}$, then there is a $3$-Lie algebra structure over the ring $\mathbb{K}[[t]]$ on $\mathfrak{g}[[t]]$ given
by
\begin{equation}\label{eq:power3-lie}
\Big[\sum_{i=0}^{+\infty}x_it^{i},\sum_{j=0}^{+\infty}y_jt^{j},\sum_{k=0}^{+\infty}z_kt^{k}\Big]_\mathfrak{g}=\sum_{s=0}^{+\infty}\sum_{i+j+k=s}[x_i,y_j,z_k]_\mathfrak{g}t^{s},\;\forall x_i,y_j,z_k\in \mathfrak{g}.
\end{equation}
For any representation $(V,\rho)$ of $(\mathfrak{g}, [\cdot,\cdot, \cdot]_\mathfrak{g})$, there is a natural representation of the $3$-Lie algebra
$\mathfrak{g}[[t]]$ on the $\mathbb{K}[[t]]$-module $V[[t]]$, which is given by
\begin{equation}
\rho\Big(\sum_{i=0}^{+\infty}x_it^{i},\sum_{j=0}^{+\infty}y_jt^{j}\Big)\Big(\sum_{k=0}^{+\infty}V_kt^{k}\Big)=\sum_{s=0}^{+\infty}\sum_{i+j+k=s}\rho(x_i,y_j)v_kt^{s},\;\forall x_i,y_j\in \mathfrak{g},\;v_k\in V.
\end{equation}
Similarly, the $2$-cocycle $\Theta$ can be extended to a $2$-cocycle ( denoted by the same notation $\Theta$) on the $3$-Lie algebra $\mathfrak{g}[[t]]$ with
coefficients in $V[[t]]$. Consider a power series
\begin{equation}\label{eq:formal1}
T_t=\sum_{i=0}^{+\infty}T_it^{i},\;T_i\in Hom_{\mathbb{K}}(V;\mathfrak{g}),
\end{equation}
that is, $T_t\in Hom_{\mathbb{K}}(V;\mathfrak{g})[[t]]=Hom_{\mathbb{K}}(V;\mathfrak{g}[[t]])$. Extend it to be a $\mathbb{K}[[t]]$-module map from
$V[[t]]$ to $\mathfrak{g}[[t]]$ which is still denoted by $T_t$.
\begin{defi}
If $T_t=\displaystyle\sum_{i=0}^{+\infty}T_it^{i}$ with $T_0=T$ satisfies
\begin{align}\label{eq:formal2}
&[T_t(u),T_t(v),T_t(w)]_\mathfrak{g}\nonumber\\
&=T_t\Big(\rho(T_t(u),T_t(v))w+\rho(T_t(v),T_t(w))u+\rho(T_t(w),T_t(u))v+\Theta(T_t(u),T_t(v),T_t(w))\Big),
\end{align}
we say that $T_t$ is a formal deformation of the $\Theta$-twisted $\mathcal{O}$-operator $T$.
\end{defi}
Recall that a formal deformation of a $3$-Lie algebra $(\mathfrak{g}, [\cdot,\cdot, \cdot]_\mathfrak{g})$ is a formal power series $\omega_t=\displaystyle\sum_{k=0}^{+\infty}\omega_kt^{k}$
where $\omega_k\in Hom(\wedge^{3}\mathfrak{g}; \mathfrak{g})$ such that $\omega_0(x, y, z) = [x, y, z]_\mathfrak{g}$ for any $x, y, z \in \mathfrak{g}$ and $\omega_t$ defines a $3$-Lie
algebra structure over the ring $\mathbb{K}[[t]]$ on $\mathfrak{g}[[t]]$.\\
Based on the relationship between $\Theta$-twisted $\mathcal{O}$-operators and $3$-Lie algebras, we have
\begin{pro}
Let  $T_t=\displaystyle\sum_{i=0}^{+\infty}T_it^{i}$ be a formal deformation of a $\Theta$-twisted $\mathcal{O}$-operator $T$ on the $3$-Lie algebra
$(\mathfrak{g}, [\cdot,\cdot, \cdot]_\mathfrak{g})$ with respect to a representation $(V,\rho)$. Then $[\cdot,\cdot,\cdot]_{T_t}$ defined by
\begin{align*}
\quad\;[u,v,w]_{T_t}=&\displaystyle\sum_{s=0}^{+\infty}\Big(\sum_{i+j=s}\Big(\rho(T_i(u),T_j(v))w+\rho(T_i(v),T_j(w))u+\rho(T_i(w),T_j(u))v\Big)\\
&+\sum_{i+j+k=s}\Theta(T_i(u),T_j(v),T_k(w))\Big)t^{s}
\end{align*} for all $u,v,w\in V$,
is a formal deformation of the associated $3$-Lie algebra $(V,[\cdot,\cdot,\cdot]_T)$ given by \eqref{eq:cro1}.
\end{pro}
By applying Eqs.\eqref{eq:power3-lie}-\eqref{eq:formal1} to expand Eq.\eqref{eq:formal2} and collecting coefficients of $t^{s}$,
we see that Eq.\eqref{eq:formal2} is equivalent to the system of equations
\begin{align}\label{eq:system}
&\sum_{i+j+k=s}[T_i(u),T_j(v),T_k(w)]_\mathfrak{g}\nonumber\\
&=\sum_{i+j+k=s}T_i\Big(\rho(T_j(u),T_k(v))w+\rho(T_j(v),T_k(w))u+\rho(T_j(w),T_k(u))v\Big)\nonumber\\
&+\sum_{i+j+k+m=s}T_i\Big(\Theta(T_j(u),T_k(v),T_m(w))\Big).
\end{align}
Note that \eqref{eq:system} holds for $n=0$ as $T_0=T$ is a $\Theta$-twisted $\mathcal{O}$-operator. For $n=1$, we get
\begin{align*}
&[Tu,Tv,T_1(w)]_\mathfrak{g}+[Tu,T_1(v),Tw]_\mathfrak{g}+[T_1(u),Tv,Tw]_\mathfrak{g}\\
&=T\Big(\rho(Tu,T_1(v))w+\rho(T_1(u),Tv)w+\rho(Tv,T_1(w))u+\rho(T_1(v),Tw)u\\
&+\rho(Tw,T_1(u))v+\rho(T_1(w),Tu)v+\Theta(Tu,Tv,T_1(w))+\Theta(Tu,T_1(v),Tw)\\
&+\Theta(T_1(u),Tv,Tw)\Big)\\
&+T_1\Big(\rho(Tu,Tv)w+\rho(Tv,Tw)u+\rho(Tw,Tu)v+\Theta(Tu,Tv,Tw)\Big),
\end{align*}
This implies that $(D_{\Theta} (T_1))(u, v,w) = 0$. Hence the linear term $T_1$ is a $1$-cocycle in the cohomology of $T$, called
the infinitesimal of the deformation $T_t$.\\

In the sequel, we discuss equivalent formal deformations.
\begin{defi}
Let $T_t=\displaystyle\sum_{i=0}^{+\infty}T_it^{i}$ and $T^{'}_t=\displaystyle\sum_{i=0}^{+\infty}T{'}_it^{i}$ be two formal deformations of a $\Theta$-twisted $\mathcal{O}$-operator $T=T_0=T^{'}_0$ on a $3$-Lie algebra $\mathfrak{g}$ with respect to a representation $(V,\rho)$. They are said to be equivalent if there exists
an element $\mathfrak{X}\in \mathfrak{g}\wedge \mathfrak{g}$, $\phi_i\in gl(\mathfrak{g})$ and $\psi_i\in gl(V)$, $i\geq2$, such that the pair
\begin{equation}\label{eq:equi}
\Big(\phi_t=Id_\mathfrak{g}+t[\mathfrak{X},-]_\mathfrak{g}+\sum_{i=2}^{+\infty}\phi_it^{i},\;\psi_t=Id_V+t\big(\rho(\mathfrak{X})(-)+\Theta(\mathfrak{X},T-)\big)+\sum_{i=2}^{+\infty}\psi_it^{i}\Big),
\end{equation}
is a morphism of a $\Theta$-twisted $\mathcal{O}$-operators from $T_t$ to $T^{'}_t$.

A $\Theta$-twisted $\mathcal{O}$-operator $T$ is said to be rigid if any
formal deformation of $T$ is equivalent to the undeformed deformation $T^{'}_t = T$ .
\end{defi}
\begin{thm}
If two formal deformations of a $\Theta$-twisted $\mathcal{O}$-operator $T$ on a $3$-Lie algebra
$(\mathfrak{g}, [\cdot,\cdot, \cdot]_\mathfrak{g})$ with respect to a representation $(V,\rho)$ are equivalent, then their infinitesimals
are in the same cohomology class.
\end{thm}
\begin{proof} Let $(\phi_t,\psi_t)$ be the two maps defined by Eq.\eqref{eq:equi} which gives an equivalence between
two deformations $T_t=\displaystyle\sum_{i=0}^{+\infty}T_it^{i}$ and $T^{'}_t=\displaystyle\sum_{i=0}^{+\infty}T{'}_it^{i}$ of a $\Theta$-twisted $\mathcal{O}$-operator $T$. By
$(\phi_t \circ T_t)(u)=(T^{'}_t\circ \psi_t)(u)$, we have
\begin{align*}
&T_1(u)=T^{'}_1(u)+T(\rho(\mathfrak{X})u+\Theta(\mathfrak{X},Tu))+[\mathfrak{X},Tu]_\mathfrak{g}\\
&=T^{'}_1(u)+(D_{\Theta}(\mathfrak{X}))(u),\;\forall u\in V,
\end{align*}
which implies that $T_1$ and $T^{'}_1$ are in the same cohomology class.\end{proof}

\section{$3$-NS-Lie algebras}
%%%%%%%%%%%%%%%%%%%%%
In this section, we introduce the notion of $3$-NS-Lie algebra as the underlying structure of twisted
$\mathcal{O}$-operators. It is a natural generalization of NS-Lie algebras introduced in \cite{Das1}. Here we study some properties of $3$-NS-Lie algebras.
\begin{defi}
A $3$-NS-Lie algebra is a vector space $A$ together with trilinear operations $\{\cdot,\cdot,\cdot\},[\![\cdot,\cdot,\cdot]\!]:\otimes^{3}A\rightarrow A$ satisfying
the following identities
\begin{equation}\label{eq:Ns}
\{x_1,x_2,x_3\}=-\{x_2,x_1,x_3\},
\end{equation}
\begin{equation}\label{eq:Nss}
[\![x_{\sigma(1)},x_{\sigma(2)},x_{\sigma(3)}]\!]=\epsilon(\sigma)[\![x_1,x_2,x_3]\!],\;\forall \sigma \in S_3,
\end{equation}
\begin{align}\label{eq:NS1}
&\quad\;\{x_1,x_2,\{x_3,x_4,x_5\}\}=\{\{x_1,x_2,x_3\}^C,x_4,x_5\}+\{[\![x_1,x_2,x_3]\!],x_4,x_5\}\nonumber\\
&+\{x_3,\{x_1,x_2,x_4\}^C,x_5\}+\{x_3,[\![x_1,x_2,x_4]\!],x_5\}+\{x_3,x_4,\{x_1,x_2,x_5\}\},
\end{align}
\begin{align}\label{eq:NS2}
&\quad\;\{\{x_1,x_2,x_3\}^C,x_4,x_5\}+\{[\![x_1,x_2,x_3]\!],x_4,x_5\}\nonumber\\
&=\{x_1,x_2,\{x_3,x_4,x_5\}\}+\{x_2,x_3,\{x_1,x_4,x_5\}\}+\{x_3,x_1,\{x_2,x_4,x_5\}\},
\end{align}
\begin{align}\label{eq:NS3}
&\quad\;[\![x_1,x_2,[x_3,x_4,x_5]_{*}]\!]+\{x_1,x_2,[\![x_3,x_4,x_5]\!]\}\nonumber\\
&=[\![[x_1,x_2,x_3]_{*},x_4,x_5]\!]+[\![x_3,[x_1,x_2,x_4]_{*},x_5]\!]+[\![x_3,x_4,[x_1,x_2,x_5]_{*}]\!]\nonumber\\
&+\{x_4,x_5,[\![x_1,x_2,x_3]\!]\}+\{x_5,x_3,[\![x_1,x_2,x_4]\!]\}+\{x_3,x_4,[\![x_1,x_2,x_5]\!]\},
\end{align}
for all $x_i\in A,\;1\leq i\leq 5$.\\ Here $\{x_1,x_2,x_3\}^C=\circlearrowleft_{x_1,x_2,x_3}\{x_1,x_2,x_3\}$ and
$[x_1,x_2,x_3]_{*}=\{x_1,x_2,x_3\}^C+[\![x_1,x_2,x_3]\!]$.
\end{defi}

\begin{rmk}
If the trilinear operation $\{\cdot,\cdot,\cdot\}$ in the above definition is trivial, one gets that $(A,[\![\cdot,\cdot,\cdot]\!])$ is a
$3$-Lie algebra. On the other hand, if $[\![\cdot,\cdot,\cdot]\!]$ is trivial then $(A,\{\cdot,\cdot,\cdot\})$ becomes a $3$-pre-Lie algebra. Thus, $3$-NS-Lie algebras
are a generalization of both $3$-Lie algebras and $3$-pre-Lie algebras (For more details about $3$-pre-Lie algebras, see \cite{BGS}).\\
In the following, we show that $3$-NS-Lie algebras split $3$-Lie algebras.
\end{rmk}

\begin{pro}
Let $(A,\{\cdot,\cdot,\cdot\},[\![\cdot,\cdot,\cdot]\!])$ be a $3$-NS-Lie algebra. Then   $(A,[\cdot,\cdot,\cdot]_*)$
is a $3$-Lie algebra.
\end{pro}
\begin{proof} By Eqs.\eqref{eq:Ns} and \eqref{eq:Nss}, then $[\cdot, \cdot, \cdot]_*$ is skew-symmetric.\\
For any $x_i\in A,\;1\leq i\leq 5$, we have
\begin{align*}
&[x_1,x_2,[x_3,x_4,x_5]_*]_*-[[x_1,x_2,x_3]_*,x_4,x_5]_*-[x_3,[x_1,x_2,x_4]_*,x_5]_*-[x_3,x_4,[x_1,x_2,x_5]_*]_*\\
&=\{x_1,x_2,\{x_3,x_4,x_5\}^{C}\}^{C}-\{\{x_1,x_2,x_3\}^{C},x_4,x_5\}^{C}-\{x_3,\{x_1,x_2,x_4\}^{C},x_5\}^{C}\\
&-\{x_3,x_4,\{x_1,x_2,x_5\}^{C}\}^{C}+\{x_1,x_2,[\![x_3,x_4,x_5]\!]\}+\{x_2,[\![x_3,x_4,x_5]\!],x_1\}\\
&+\{[\![x_3,x_4,x_5]\!],x_1,x_2\}+[\![x_1,x_2,[x_3,x_4,x_5]_*]\!]-\{[\![x_1,x_2,x_3]\!],x_4,x_5\}\\
&-\{x_4,x_5,[\![x_1,x_2,x_3]\!]\}-\{x_5,[\![x_1,x_2,x_3]\!],x_4\}-[\![[x_1,x_2,x_3]_*,x_4,x_5]\!]\\
&-\{x_3,[\![x_1,x_2,x_4]\!],x_5\}-\{[\![x_1,x_2,x_4]\!],x_5,x_3\}-\{x_5,x_3,[\![x_1,x_2,x_4]\!]\}\\
&-[\![x_3,[x_1,x_2,x_4]_*,x_5]\!]-\{x_3,x_4,[\![x_1,x_2,x_5]\!]\}-\{x_4,[\![x_1,x_2,x_5]\!],x_3\}\\
&-\{[\![x_1,x_2,x_5]\!],x_3,x_4\}-[\![x_3,x_4,[x_1,x_2,x_5]_*]\!]\\
&=\{x_1,x_2,\{x_3,x_4,x_5\}\} + \{x_1,x_2,\{x_4,x_5,x_3\}\} + \{x_1,x_2,\{x_5,x_3,x_4\}\}\\
&+\{x_2,\{x_3,x_4,x_5\}^{C},x_1\} + \{\{x_3,x_4,x_5\}^{C},x_1,x_2\}\\
&-\{x_4,x_5,\{x_1,x_2,x_3\}\} - \{x_4,x_5,\{x_2,x_3,x_1\}\} - \{x_4,x_5,\{x_3,x_1,x_2\}\}\\
&-\{\{x_1,x_2,x_3\}^{C},x_4,x_5\} - \{x_5,\{x_1,x_2,x_3\}^{C},x_4\}\\
&-\{x_5,x_3,\{x_1,x_2,x_4\}\} - \{x_5,x_3,\{x_2,x_4,x_1\}\} - \{x_5,x_3,\{x_4,x_1,x_2\}\}\\
&-\{x_3,\{x_1,x_2,x_4\}^{C},x_5\} - \{\{x_1,x_2,x_4\}^{C},x_5,x_3\}\\
&-\{x_3,x_4,\{x_1,x_2,x_5\}\} - \{x_3,x_4,\{x_2,x_5,x_1\}\} - \{x_3,x_4,\{x_5,x_1,x_2\}\}\\
&-\{x_4,\{x_1,x_2,x_5\}^{C},x_3\} - \{\{x_1,x_2,x_5\}^{C},x_3,x_4\}\\
&+\{x_1,x_2,[\![x_3,x_4,x_5]\!]\}+\{x_2,[\![x_3,x_4,x_5]\!],x_1\}+\{[\![x_3,x_4,x_5]\!],x_1,x_2\}\\
&+[\![x_1,x_2,[x_3,x_4,x_5]_*]\!]-\{[\![x_1,x_2,x_3]\!],x_4,x_5\}-\{x_4,x_5,[\![x_1,x_2,x_3]\!]\}\\
&-\{x_5,[\![x_1,x_2,x_3]\!],x_4\}-[\![[x_1,x_2,x_3]_*,x_4,x_5]\!]-\{x_3,[\![x_1,x_2,x_4]\!],x_5\}\\
&-\{[\![x_1,x_2,x_4]\!],x_5,x_3\}-\{x_5,x_3,[\![x_1,x_2,x_4]\!]\}-[\![x_3,[x_1,x_2,x_4]_*,x_5]\!]\\
&-\{x_3,x_4,[\![x_1,x_2,x_5]\!]\}-\{x_4,[\![x_1,x_2,x_5]\!],x_3\}-\{[\![x_1,x_2,x_5]\!],x_3,x_4\}\\
&-[\![x_3,x_4,[x_1,x_2,x_5]_*]\!]=0.
\end{align*}
This is because
\begin{align*}
&\{x_1,x_2,\{x_3,x_4,x_5\}\} = \{\{x_1,x_2,x_3\}^{C},x_4,x_5\} + \{x_3,\{x_1,x_2,x_4\}^{C},x_5\}\\
&\quad\quad\quad\quad + \{x_3,x_4,\{x_1,x_2,x_5\}\}+\{[\![x_1,x_2,x_3]\!],x_4,x_5\}+\{x_3,[\![x_1,x_2,x_4]\!],x_5\},\\
&\{x_1,x_2,\{x_4,x_5,x_3\}\} = \{\{x_1,x_2,x_4\}^{C},x_5,x_3\} + \{x_4,\{x_1,x_2,x_5\}^{C},x_3\}\\
& \quad\quad\quad\quad+ \{x_4,x_5,\{x_1,x_2,x_3\}\}+\{[\![x_1,x_2,x_4]\!],x_5,x_3\}+\{x_4,[\![x_1,x_2,x_5]\!],x_3\},\\
&\{x_1,x_2,\{x_5,x_3,x_4\}\} = \{x_5,\{x_1,x_2,x_3\}^{C},x_4\} + \{\{x_1,x_2,x_5\}^{C},x_3,x_4\}\\
& \quad\quad\quad\quad+ \{x_5,x_3,\{x_1,x_2,x_4\}\}+\{x_5,[\![x_1,x_2,x_3]\!],x_4\}+\{[\![x_1,x_2,x_5]\!],x_3,x_4\},\\
&\{x_2,\{x_3,x_4,x_5\}^{C},x_1\}+\{x_2,[\![x_3,x_4,x_5]\!],x_1\} = \{x_4,x_5,\{x_2,x_3,x_1\}\}\\
&\quad\quad\quad\quad+ \{x_5,x_3,\{x_2,x_4,x_1\}\}+ \{x_3,x_4,\{x_2,x_5,x_1\}\},\\
&\{\{x_3,x_4,x_5\}^{C},x_1,x_2\}+\{[\![x_3,x_4,x_5]\!],x_1,x_2\} = \{x_3,x_4,\{x_5,x_1,x_2\}\} \\
&\quad\quad\quad\quad+ \{x_4,x_5,\{x_3,x_1,x_2\}\}+ \{x_5,x_3,\{x_4,x_1,x_2\}\},\\
&[\![x_1,x_2,[x_3,x_4,x_5]_{*}]\!]+\{x_1,x_2,[\![x_3,x_4,x_5]\!]\}\\
&\quad\quad\quad\quad=[\![[x_1,x_2,x_3]_{*},x_4,x_5]\!]+[\![x_3,[x_1,x_2,x_4]_{*},x_5]\!]+[\![x_3,x_4,[x_1,x_2,x_5]_{*}]\!]\\
&\quad\quad\quad\quad+\{x_4,x_5,[\![x_1,x_2,x_3]\!]\}+\{x_5,x_3,[\![x_1,x_2,x_4]\!]\}+\{x_3,x_4,[\![x_1,x_2,x_5]\!]\}.
\end{align*}
Thus the proof is completed.\end{proof}
\noindent The $3$-Lie algebra $(A,[\cdot,\cdot,\cdot]_*)$ of the above proposition is called the subadjacent $3$-Lie algebra of $(A,\{\cdot,\cdot,\cdot \},[\![\cdot,\cdot,\cdot]\!])$
and $(A,\{\cdot,\cdot,\cdot\},[\![\cdot,\cdot,\cdot]\!])$ is called a compatible $3$-NS-Lie algebra on $(A,[\cdot,\cdot,\cdot]_*)$.
\begin{pro}\label{3nslie}
Let $(\mathfrak{g},[\cdot,\cdot,\cdot]_\mathfrak{g})$ be a $3$-Lie algebra and $N:\mathfrak{g}\rightarrow \mathfrak{g}$ be a Nijenhuis operator. Then
\begin{equation}
\{x,y,z\}=[Nx,Ny,z]_\mathfrak{g},
\end{equation}
\begin{equation}
[\![x,y,z]\!]=-N\big([Nx,y,z]_\mathfrak{g}+[x,Ny,z]_\mathfrak{g}+[x,y,Nz]_\mathfrak{g}-N[x,y,z]_\mathfrak{g}\big),
\end{equation}
defines a $3$-NS-Lie algebra on $\mathfrak{g}$.
\end{pro}
\begin{proof} Let $x,y,z\in \mathfrak{g}$. It is obvious that
\begin{align*}
\{x,y,z\}=[Nx,Ny,z]_\mathfrak{g}=-[Ny,Nx,z]_\mathfrak{g}=-\{y,x,z\},
\end{align*}
and
\begin{align*}
&[\![x,y,z]\!]=-N\big([Nx,y,z]_\mathfrak{g}+[x,Ny,z]_\mathfrak{g}+[x,y,Nz]_\mathfrak{g}-N[x,y,z]_\mathfrak{g}\big)\\
&-\big(-N\big([Ny,x,z]_\mathfrak{g}+[y,Nx,z]_\mathfrak{g}+[y,x,Nz]_\mathfrak{g}-N[y,x,z]_\mathfrak{g}\big)\big)=-[\![y,x,z]\!].
\end{align*}
Similarly, we get $[\![x,y,z]\!]=-[\![x,z,y]\!]$. For $x_1,x_2,x_3,x_4,x_5\in \mathfrak{g}$, we have
\begin{align*}
&\{x_1,x_2,\{x_3,x_4,x_5\}\}-\{\{x_1,x_2,x_3\}^C,x_4,x_5\}-\{[\![x_1,x_2,x_3]\!],x_4,x_5\}\\
&-\{x_3,\{x_1,x_2,x_4\}^C,x_5\}-\{x_3,[\![x_1,x_2,x_4]\!],x_5\}-\{x_3,x_4,\{x_1,x_2,x_5\}\}\\
&=[Nx_1,Nx_2,[Nx_3,Nx_4,x_5]_\mathfrak{g}]_\mathfrak{g}\\
&-[N[Nx_1,Nx_2,x_3]_\mathfrak{g}+N[Nx_2,Nx_3,x_1]_\mathfrak{g}+N[Nx_3,Nx_1,x_2]_\mathfrak{g},Nx_4,x_5]_\mathfrak{g}\\
&-[N\big(-N\big([Nx_1,x_2,x_3]_\mathfrak{g}+[x_1,Nx_2,x_3]_\mathfrak{g}+[x_1,x_2,Nx_3]_\mathfrak{g}-N[x_1,x_2,x_3]_\mathfrak{g}\big)\big),Nx_4,x_5]_\mathfrak{g}\\
&-[Nx_3,N[Nx_1,Nx_2,x_4]_\mathfrak{g}+N[Nx_2,Nx_4,x_1]_\mathfrak{g}+N[Nx_4,Nx_1,x_2]_\mathfrak{g},x_5]_\mathfrak{g}\\
&-[Nx_3,N\big(-N\big([Nx_1,x_2,x_4]_\mathfrak{g}+[x_1,Nx_2,x_4]_\mathfrak{g}+[x_1,x_2,Nx_4]_\mathfrak{g}-N[x_1,x_2,x_4]_\mathfrak{g}\big)\big),x_5]_\mathfrak{g}\\
&-[Nx_3,Nx_4,[Nx_1,Nx_2,x_5]_\mathfrak{g}]_\mathfrak{g}\\
^{\eqref{eq:nijenhuis}}&=[Nx_1,Nx_2,[Nx_3,Nx_4,x_5]_\mathfrak{g}]_\mathfrak{g}-[[Nx_1,Nx_2,Nx_3]_\mathfrak{g},Nx_4,x_5]_\mathfrak{g}\\
&-[Nx_3,[Nx_1,Nx_2,Nx_4]_\mathfrak{g},x_5]_\mathfrak{g}-[Nx_3,Nx_4,[Nx_1,Nx_2,x_5]_\mathfrak{g}]_\mathfrak{g}\\
^{\eqref{eq:de1}}&=0.
\end{align*}
Then Eq.\eqref{eq:NS1} holds. One may prove similarly Eq.\eqref{eq:NS2} and  Eq.\eqref{eq:NS3}, we leave details to readers. This completes the proof.\end{proof}
\begin{ex}
 Let $(\mathfrak{g},[\cdot,\cdot,\cdot]_\mathfrak{g})$ be a $3$-dimensional $3$-Lie algebra
with a basis $\{e_1,e_2,e_3\}$ whose nonzero
brackets are given as follows:
\begin{equation*}
    [e_1,e_2,e_3]_{\mathfrak{g}}=e_2.
\end{equation*}
Let $N:\mathfrak{g}\rightarrow \mathfrak{g}$ be a linear map given with respect to the basis by
\begin{align*}
 N=\left(
                     \begin{array}{ccc}
                       d & 0 & 0 \\
                       0 & c & f \\
                       0 & 0 & c \\
                     \end{array}
                   \right),
 \end{align*}
where $d,\ c,\ f$ are  parameters such that $dc\neq 0$. Then we have 
 \begin{equation*}
     [Ne_1,Ne_2,Ne_3]_{\mathfrak{g}}=dc^{2}e_2.
 \end{equation*}
 Furthermore, we have
 \begin{align*}
     &[Ne_1,Ne_2,e_3]_{\mathfrak{g}}=dce_2,\quad [Ne_1,e_2,e_3]_{\mathfrak{g}}=de_2,\\
      &[Ne_1,e_2,Ne_3]_{\mathfrak{g}}=dce_2,\quad [e_1,Ne_2,e_3]_{\mathfrak{g}}=ce_2,\\
       &[e_1,Ne_2,Ne_3]_{\mathfrak{g}}=c^{2}e_2,\quad\; [e_1,e_2,Ne_3]_{\mathfrak{g}}=ce_2.\\
 \end{align*}
 It is straightforward to deduce that
 \begin{align*}
&[Ne_1,Ne_2,Ne_3]_\mathfrak{g}=N\Big([Ne_1,Ne_2,e_3]_\mathfrak{g}+[Ne_1,e_2,Ne_3]_\mathfrak{g}+[e_1,Ne_2,Ne_3]_\mathfrak{g}\nonumber\\
&-N([Ne_1,e_2,e_3]_\mathfrak{g}+[e_1,Ne_2,e_3]_\mathfrak{g}+[e_1,e_2,Ne_3]_\mathfrak{g})+N^{2}[e_1,e_2,e_3]_\mathfrak{g}\Big).
\end{align*}
Thus, $N$ is a Nijenhuis operator on $(\mathfrak{g},[\cdot,\cdot,\cdot]_\mathfrak{g})$. By Proposition \ref{3nslie}, we get that $(\mathfrak{g},\{\cdot,\cdot,\cdot\},[\![\cdot,\cdot,\cdot]\!])$ is a $3$-dimensional $3$-NS-Lie algebra with a basis $\{e_1, e_2, e_3\}$ whose nonzero
brackets are given as follows:
\begin{equation*}
    \{e_1,e_2,e_3\}=dce_2,\quad [\![e_1,e_2,e_3]\!]=-(dc+c^{2})e_2.
\end{equation*}
\end{ex}
Let $(A,\{\cdot,\cdot,\cdot\},[\![\cdot,\cdot,\cdot]\!])$ be an $3$-NS-Lie algebra. Define the left multiplication $L:\wedge^{2}A\rightarrow gl(A)$ by
$L(x,y)z=\{x,y,z\}$ and $\Theta:\wedge^{3}A\rightarrow A$ by $\Theta(x,y,z)=[\![x,y,z]\!]$, for $x,y,z\in A$. With these notations, we have the following proposition.
\begin{pro}
Let $(A,\{\cdot,\cdot,\cdot\},[\![\cdot,\cdot,\cdot]\!])$ be a $3$-NS-Lie algebra. Then $(A,L)$ is a representation of the subadjacent $3$-Lie algebra $(A,[\cdot,\cdot,\cdot]_*)$
and $\Theta$ defined above is a $2$-cocycle. Moreover, the identity map $Id:A\rightarrow A$ is a $\Theta$-twisted $\mathcal{O}$-operator on the
$3$-Lie algebra $(A,[\cdot,\cdot,\cdot]_*)$ with respect to $(A,L)$.
\end{pro}
\begin{proof} For any $x_i\in A,\;1\leq i\leq 5$, we have
\begin{align*}
&\quad\;L([x_1,x_2,x_3]_*,x_4)x_5+L(x_3,[x_1,x_2,x_4]_*)x_5=\{[x_1,x_2,x_3]_*,x_4,x_5\}+\{x_3,[x_1,x_2,x_4]_*,x_5\}\\
^{\eqref{eq:NS1}}=&\quad \;\{x_1,x_2,\{x_3,x_4,x_5\}\}-\{x_3,x_4,\{x_1,x_2,x_5\}\}=[L(x_1,x_2),L(x_3,x_4)]_{End(A)}(x_5).
\end{align*}
Similarly, by Eq.\eqref{eq:NS2}
\begin{align*}
&\quad\;L([x_1,x_2,x_3]_*,x_4)x_5=L(x_1,x_2)\circ L(x_3,x_4)x_5\\
&+L(x_2,x_3)\circ L(x_1,x_4)x_5+L(x_3,x_1)\circ L(x_2,x_4)x_5.
\end{align*}
Therefore, $(A,L)$ is a representation of $(A,[\cdot,\cdot,\cdot]_*)$. Moreover,  Condition Eq.\eqref{eq:NS3} is equivalent
to $\Theta$ is a $2$-cocycle in the Chevalley-Eilenberg cochain complex of the $3$-Lie algebra $(A,[\cdot,\cdot,\cdot]_*)$ with coefficients
in the representation $(A,L)$. Finally, we have
\begin{align*}
&\quad\;Id\big(L(Idx,Idy)z+L(Idy,Idz)x+L(Idz,Idx)y+\Theta(Idx,Idy,Idz)\big)\\
&=\{x,y,z\}+\{y,z,x\}+\{z,x,y\}+[\![x,y,z]\!]\\
&=[x,y,z]_*=[Idx,Idy,Idz]_*,
\end{align*}
which shows that $Id:A\rightarrow A$ is a $\Theta$-twisted Rota-Baxter operator on the $3$-Lie algebra $(A,[\cdot,\cdot,\cdot]_*)$
with respect to the representation $(A,L)$.\end{proof}
\begin{pro}\label{eq:onV}
Let $(\mathfrak{g},[\cdot,\cdot,\cdot]_\mathfrak{g})$ be a $3$-Lie algebra, $(V,\rho)$ be a representation and $\Theta \in\mathfrak{C}_{3Lie}^{2}(\mathfrak{g};V)$ be a $2$-cocycle. Let
$T:V\rightarrow \mathfrak{g}$ be a $\Theta$-twisted $\mathcal{O}$-operator. Then there is a $3$-NS-Lie algebra structure on $V$ given by
\begin{equation}
\{u,v,w\}=\rho(Tu,Tv)w,\;[\![u,v,w]\!]=\Theta(Tu,Tv,Tw),\;\forall u,v,w\in V.
\end{equation}
\end{pro}
\begin{proof} Let $u, v,w\in V$. It is obvious that
\begin{equation*}
\{u, v,w\} = \rho(Tu, Tv)w = -\rho(Tv, Tu)w = -\{v, u,w\},
\end{equation*}
and $[\![u, v,w]\!]$ is skew-symmetric since $\Theta$ is a $2$-cocycle.
Furthermore, the following equation holds.
\begin{equation*}
[u, v,w]_* = \rho(Tu, Tv)w + \rho(Tv, Tw)u + \rho(Tw, Tu)v +\Theta(Tu,Tv,Tw).
\end{equation*}
Since $T$ is a $\Theta$-twisted $\mathcal{O}$-operator, we have
\begin{equation*}
T[u, v,w]_* = [Tu, Tv, Tw]_\mathfrak{g}.
\end{equation*}
For $u_i\in V,\;1\leq i \leq5$, we have
\begin{align*}
&\{[u_1,u_2,u_3]_*,u_4,u_5\}=\rho([Tu_1,Tu_2,Tu_3]_*,Tu_4)u_5\\
&=\rho(Tu_1,Tu_2)\rho(Tu_3,Tu_4)u_5+\rho(Tu_1,Tu_3)\rho(Tu_1,Tu_4)u_5+\rho(Tu_3,Tu_1)\rho(Tu_2,Tu_4)u_5\\
&=\{[u_1,u_2,u_3]_*,u_4,u_5\}+\{u_3,[u_1,u_2,u_4]_*,u_5\}+\{u_3,u_4,\{u_1,u_2,u_5\}\}.
\end{align*}
Similarly,
\begin{align*}
&\{u_1,u_2,\{u_3,u_4,u_5\}\}=\rho(Tu_1,Tu_2)\rho(Tu_3,Tu_4)u_5\\
&=\rho([Tu_1,Tu_2,Tu_3]_*,Tu_4)u_5+\rho(Tu_3,[Tu_1,Tu_2,Tu_4]_*)u_5+\rho(Tu_3,Tu_4)\rho(Tu_1,Tu_2)u_5\\
&=\{u_1,u_2,\{u_3,u_4,u_5\}\}+\{u_2,u_3,\{u_1,u_4,u_5\}\}+\{u_3,u_1,\{u_2,u_4,u_5\}\}.
\end{align*}
Since $(V, \rho)$ is a representation of $A$, Eq.\eqref{eq:NS1} and \eqref{eq:NS2} holds.
\begin{align*}
&\{u_1,u_2,[u_3,u_4,u_5]_*\}+\{u_1,u_2,[\![u_3,u_4,u_5]\!]\}\\
&=\Theta(Tu_1,Tu_2,[Tu_3,Tu_3,Tu_5]_\mathfrak{g})+\rho(Tu_1,Tu_2)\Theta(Tu_3,Tu_4,Tu_5)\\
&=\Theta([Tu_1,Tu_2,Tu_3]_\mathfrak{g},Tu_4,Tu_5)+\Theta(Tu_3,[Tu_1,Tu_2,Tu_4]_\mathfrak{g},Tu_5)\\
&+\Theta(Tu_3,Tu_4,[Tu_1,Tu_2,Tu_5]_\mathfrak{g})+\rho(Tu_4,Tu_5)\Theta(Tu_1,Tu_2,Tu_3)\\
&+\rho(Tu_5,Tu_3)\Theta(Tu_1,Tu_2,Tu_4)+\rho(Tu_3,Tu_4)\Theta(Tu_1,Tu_2,Tu_5)\\
&=[\![[u_1,u_2,u_3]_{*},u_4,u_5]\!]+[\![u_3,[u_1,u_2,u_4]_{*},u_5]\!]+[\![u_3,u_4,[u_1,u_2,u_5]_{*}]\!]\\
&+\{u_4,u_5,[\![u_1,u_2,u_3]\!]\}+\{u_5,u_3,[\![u_1,u_2,u_4]\!]\}+\{u_3,u_4,[\![u_1,u_2,u_5]\!]\},
\end{align*}
since $\Theta$ is a $2$-cocycle, Eq.\eqref{eq:NS3} holds. This completes the proof.\end{proof}
\begin{rmk}There is a pair of adjoint functors between the category of $3$-NS-Lie algebras and $3$-Lie algebras together with  $\Theta$-twisted $\mathcal{O}$-operators.
\end{rmk}
In the following, we give a necessary and sufficient condition for the existence of a
compatible $3$-NS-Lie algebra structure on a $3$-Lie algebra.
\begin{pro}
 Let $(\mathfrak{g}, [\cdot,\cdot,\cdot]_\mathfrak{g})$ be a $3$-Lie algebra. Then there is a compatible $3$-NS-Lie
algebra structure on $\mathfrak{g}$ if and only if there exists an invertible $\Theta$-twisted $\mathcal{O}$-operator $T : V \rightarrow \mathfrak{g}$ on $\mathfrak{g}$ with respect to a representation $(V, \rho)$ and a $2$-cocycle $\Theta$.
Furthermore, the compatible  $3$-NS-Lie algebra structure on $\mathfrak{g}$ is given by
\begin{equation}
\{x,y,z\}=T(\rho(x,y)T^{-1}(z)),\;\;[\![x,y,z]\!]=T(\Theta(x,y,z)),\;for\;x,y,z\in \mathfrak{g}.
\end{equation}
\end{pro}
\begin{proof} Let $T : V \rightarrow \mathfrak{g}$ be an invertible $\Theta$-twisted $\mathcal{O}$-operator on $\mathfrak{g}$
with respect to a representation $(V, \rho)$ and a $2$-cocycle $\Theta$. By Proposition \ref{eq:onV}, there
is a $3$-NS-Lie algebra structure on $V$ given by
\begin{equation*}
\{u,v,w\}=\rho(Tu,Tv)w,\;[\![u,v,w]\!]=\Theta(Tu,Tv,Tw),\;\forall u,v,w\in V.
\end{equation*}
Since $T$ is an invertible map,  operations
\begin{equation*}
\{x,y,z\}=T\{T^{-1}(x),T^{-1}(y),T^{-1}(z)\}=T(\rho(x,y)T^{-1}(z)),
\end{equation*}
\begin{equation*}
[\![x,y,z]\!]=T[\![T^{-1}(x),T^{-1}(y),T^{-1}(z)]\!]=T(\Theta(x,y,z)),\;\forall x,y,z\in \mathfrak{g},
\end{equation*}
defines a $3$-NS-Lie algebra on $\mathfrak{g}$. Moreover, we have
\begin{align*}
&\{x,y,z\}+\{y,z,x\}+\{z,x,y\}+[\![x,y,z]\!]\\
&=T\big(\rho(x,y)T^{-1}(z)+\rho(y,z)T^{-1}(x)+\rho(z,x)T^{-1}(y)+\Theta(x,y,z)\big)\\
&=T\big(\rho(T\circ T^{-1}(x),T\circ T^{-1}(y))T^{-1}(z)+\rho(T\circ T^{-1}(y),T\circ T^{-1}(z))T^{-1}(x)\\
&+\rho(T\circ T^{-1}(z),T\circ T^{-1}(x))T^{-1}(y)+\Theta(T\circ T^{-1}(x),T\circ T^{-1}(y),T \circ T^{-1}(z))\big)\\
&=[T\circ T^{-1}(x),T\circ T^{-1}(y),T \circ T^{-1}(z)]_\mathfrak{g}=[x,y,z]_\mathfrak{g}.
\end{align*}
Conversely, the identity map $Id:\mathfrak{g}\rightarrow \mathfrak{g}$ is a $\Theta$-twisted $\mathcal{O}$-operator on $(\mathfrak{g},[\cdot,\cdot,\cdot]_\mathfrak{g})$ with respect to the representation
$(\mathfrak{g},L)$. \end{proof}

\section{Twisted $\mathcal{O}$-operators on $3$-Lie algebras induced by Lie algebras}

In the following, we provide and investigate  twisted $\mathcal{O}$-operators on $3$-Lie algebras induced by Lie algebras $($see \cite{AMS0,AKMS}$)$. 
Recall that given a Lie algebra and a trace map
one can construct a $3$-Lie algebra. Let $(\mathfrak{g}, [\cdot, \cdot]_{\mathfrak{g}})$ be a Lie algebra and
$\tau : \mathfrak{g} \rightarrow \mathbb{K}$ be a  linear form. We say that $\tau$ is a trace map on $\mathfrak{g}$ if $\tau([x,y]) = 0,\;\forall x,y\in \mathfrak{g}$. For
any $x_1, x_2, x_3 \in \mathfrak{g}$, we define the ternary bracket $[\cdot,\cdot,\cdot]_\tau$ by
\begin{equation}\label{eq:induced1}
[x_1,x_2,x_3]_{\tau}=\tau(x_1)[x_2,x_3]_{\mathfrak{g}}+\tau(x_2)[x_3,x_1]_{\mathfrak{g}}+\tau(x_3)[x_1,x_2]_{\mathfrak{g}}.
\end{equation}
\begin{lem}
Let $(\mathfrak{g},[\cdot,\cdot]_{\mathfrak{g}})$ be a Lie algebra and $\tau$ be a trace map on $\mathfrak{g}$, then $(\mathfrak{g},[\cdot,\cdot,\cdot]_{\tau})$ is a $3$-Lie algebra, called  $3$-Lie algebra induced by the Lie algebra $(\mathfrak{g},[\cdot,\cdot]_{\mathfrak{g}})$.
\end{lem}
\begin{pro}[\cite{Kitouni}]
Let $(V,\rho)$ be a representation of a Lie algebra $(\mathfrak{g},[\cdot,\cdot]_{\mathfrak{g}})$ and $\tau$ be a trace map on  $\mathfrak{g}$. Then $(V,\rho_{\tau})$ is a representation of the induced $3$-Lie algebra $(\mathfrak{g},[\cdot,\cdot,\cdot]_{\tau})$, where $\rho_{\tau}:\mathfrak{g}\wedge \mathfrak{g}\rightarrow gl(V)$ is defined by
\begin{equation}\label{eq:induced2}
\rho_{\tau}(x,y)=\tau(x)\rho(y)-\tau(y)\rho(x),\;\forall x,y\in \mathfrak{g}.
\end{equation}
\end{pro}
Recall that $\Theta \in \mathfrak{C}_{CE}^2(\mathfrak{g},V)$ is a $2$-cocycle with respect to Chevalley-Eilenberg cohomology of the Lie algebra $(\mathfrak{g},[\cdot,\cdot]_{\mathfrak{g}})$ with values in a $\mathfrak{g}$-module $V$ means $\Theta : \mathfrak{g} \wedge \mathfrak{g} \rightarrow V$ is a skew-symmetric bilinear map satisfying, for any $x,y,z\in \mathfrak{g},
$\begin{align}\label{eq:teta}
&\rho(x) \Theta(y,z)+\rho(y) \Theta(z,x)+\rho(z) \Theta(x,y)+\Theta(x,[y,z]_{\mathfrak{g}})+\Theta(y,[z,x]_{\mathfrak{g}})+\Theta(z,[x,y]_{\mathfrak{g}})=0.
\end{align}
\begin{pro}\label{eq:induced3}
Let $\Theta \in \mathfrak{C}_{CE}^{2}(\mathfrak{g},V)$ be a $2$-cocycle with respect to Chevalley-Eilenberg cohomology of $(\mathfrak{g},[\cdot,\cdot]_{\mathfrak{g}})$ with coefficients
in $V$ and $\tau$ be a trace map on $\mathfrak{g}$. Then $\Theta_{\tau}(x,y,z)=\circlearrowleft_{x,y,z}\tau(x)\Theta(y,z)$ is a $2$-cocycle
of the induced $3$-Lie algebra with coefficients in $(V,\rho_{\tau})$.
\end{pro}
\begin{proof}
Let $\Theta \in \mathfrak{C}_{CE}^{2}(\mathfrak{g},V)$ be a $2$-cocycle and $\tau$ be a trace map, $(V,\rho_{\tau})$ be a representation of the induced $3$-Lie algebra, and let $\Theta_{\tau}(x,y,z)=\circlearrowleft_{x,y,z\in \mathfrak{g}}\tau(x)\Theta(y,z)$. Then we have
\begin{align*}
&\quad \Theta_{\tau}(x_1,x_2,[y_1,y_2,y_3]_{\tau})+\rho_{\tau}(x_1,x_2)\Theta_{\tau}(y_1,y_2,y_3)-
\Theta_{\tau}([x_1,x_2,y_1]_{\tau},y_2,y_3)\\
&-\Theta_{\tau}(y_1,[x_1,x_2,y_2]_{\tau},y_3)-\Theta_{\tau}(y_1,y_2,[x_1,x_2,y_3]_{\tau})-\rho_{\tau}(y_2,y_3)\Theta_{\tau}(x_1,x_2,y_1)\\
&-\rho_{\tau}(y_3,y_1)\Theta_{\tau}(x_1,x_2,y_2)-\rho_{\tau}(y_1,y_2)\Theta_{\tau}(x_1,x_2,y_3)\\
^{\eqref{eq:induced1}+\eqref{eq:induced2}}&=\tau(x_1)\tau(y_1)\Theta(x_2,[y_2,y_3]_{\mathfrak{g}})
+\tau(x_1)\tau(y_2)\Theta(x_2,[y_3,y_1]_{\mathfrak{g}})+\tau(x_1)\tau(y_3)\Theta(x_2,[y_1,y_2]_{\mathfrak{g}})\\
&+\tau(x_2)\tau(y_1)\Theta([y_2,y_3]_{\mathfrak{g}},x_1)+\tau(x_2)\tau(y_2)\Theta([y_3,y_1]_{\mathfrak{g}},x_1)
+\tau(x_2)\tau(y_3)\Theta([y_1,y_2]_{\mathfrak{g}},x_1)\\
&+\tau(x_1)\tau(y_1)\rho(x_2)\Theta(y_2,y_3)+\tau(x_1)\tau(y_2)\rho(x_2)\Theta(y_3,y_1)+\tau(x_1)\tau(y_3)\rho(x_2)\Theta(y_1,y_2)\\
&-\tau(x_2)\tau(y_1)\rho(x_1)\Theta(y_2,y_3)-\tau(x_2)\tau(y_2)\rho(x_1)\Theta(y_3,y_1)-\tau(x_2)\tau(y_3)\rho(x_1)\Theta(y_1,y_2)\\
&-\tau(y_2)\tau(x_1)\Theta(y_3,[x_2,y_1]_{\mathfrak{g}})
-\tau(y_2)\tau(x_2)\Theta(y_3,[y_1,x_1]_{\mathfrak{g}})-\tau(y_2)\tau(y_1)\Theta(y_3,[x_1,x_2]_{\mathfrak{g}})\\
&-\tau(y_3)\tau(x_1)\Theta([x_2,y_1]_{\mathfrak{g}},y_2)-\tau(y_3)\tau(x_2)\Theta([y_1,x_1]_{\mathfrak{g}},y_2)
-\tau(y_3)\tau(y_1)\Theta([x_1,x_2]_{\mathfrak{g}},y_2)\\
&-\tau(y_1)\tau(x_1)\Theta([x_2,y_2]_{\mathfrak{g}},y_3)-\tau(y_1)\tau(x_2)\Theta([y_2,x_1]_{\mathfrak{g}},y_3)
-\tau(y_1)\tau(y_2)\Theta([x_1,x_2]_{\mathfrak{g}},y_3)\\
&-\tau(y_3)\tau(x_1)\Theta(y_1,[x_2,y_2]_{\mathfrak{g}})
-\tau(y_3)\tau(x_2)\Theta(y_1,[y_2,x_1]_{\mathfrak{g}})-\tau(y_3)\tau(y_2)\Theta(y_1,[x_1,x_2]_{\mathfrak{g}})\\
&-\tau(y_1)\tau(x_1)\Theta(y_2,[x_2,y_3]_{\mathfrak{g}})
-\tau(y_1)\tau(x_2)\Theta(y_2,[y_3,x_1]_{\mathfrak{g}})-\tau(y_1)\tau(y_3)\Theta(y_2,[x_1,x_2]_{\mathfrak{g}})\\
&-\tau(y_2)\tau(x_1)\Theta([x_2,y_3]_{\mathfrak{g}},y_1)-\tau(y_2)\tau(x_2)\Theta([y_3,x_1]_{\mathfrak{g}},y_2)
-\tau(y_2)\tau(y_3)\Theta([x_1,x_2]_{\mathfrak{g}},y_2)\\
&-\tau(y_2)\tau(x_1)\rho(y_3)\Theta(x_2,y_1)-\tau(y_2)\tau(x_2)\rho(y_3)\Theta(y_1,x_1)-\tau(y_2)\tau(y_1)\rho(y_3)\Theta(x_1,x_2)\\
&+\tau(y_3)\tau(x_1)\rho(y_2)\Theta(x_2,y_1)+\tau(y_3)\tau(x_2)\rho(y_2)\Theta(y_1,x_1)+\tau(y_3)\tau(y_1)\rho(y_2)\Theta(x_1,x_2)\\
&-\tau(y_3)\tau(x_1)\rho(y_1)\Theta(x_2,y_2)-\tau(y_3)\tau(x_2)\rho(y_1)\Theta(y_2,x_1)-\tau(y_3)\tau(y_2)\rho(y_1)\Theta(x_1,x_2)\\
&+\tau(y_1)\tau(x_1)\rho(y_3)\Theta(x_2,y_2)+\tau(y_1)\tau(x_2)\rho(y_3)\Theta(y_2,x_1)+\tau(y_1)\tau(y_2)\rho(y_3)\Theta(x_1,x_2)\\
&-\tau(y_1)\tau(x_1)\rho(y_2)\Theta(x_2,y_3)-\tau(y_1)\tau(x_2)\rho(y_2)\Theta(y_3,x_1)-\tau(y_1)\tau(y_3)\rho(y_2)\Theta(x_1,x_2)\\
&+\tau(y_2)\tau(x_1)\rho(y_1)\Theta(x_2,y_3)+\tau(y_2)\tau(x_2)\rho(y_1)\Theta(y_3,x_1)+\tau(y_2)\tau(y_3)\rho(y_1)\Theta(x_1,x_2)\\
^{\eqref{eq:teta}}&= 0.
\end{align*}
\end{proof}
Let $(\mathfrak{g},[\cdot,\cdot]_{\mathfrak{g}})$ be a Lie algebra and $(V,\rho)$ be a representation. Recall from \cite{Das1}  that a $\Theta$-twisted
$\mathcal{O}$-operator on $\mathfrak{g}$ is a linear map $T:V\rightarrow \mathfrak{g}$ satisfying
\begin{equation}\label{eq:tetatwist}
[Tu,Tv]_\mathfrak{g}=T\Big(\rho(Tu)v- \rho(Tv)u +\Theta(Tu,Tv)\Big),\;\forall u,v\in V.
\end{equation}
\begin{thm}
Let $T: V\rightarrow \mathfrak{g}$ be a $\Theta$-twisted $\mathcal{O}$-operator on $(\mathfrak{g},[\cdot,\cdot]_{\mathfrak{g}})$ and $\tau$ be a trace map. Then $T: V\rightarrow \mathfrak{g}$ is a $\Theta_{\tau}$-twisted $\mathcal{O}$-operator on the induced $3$-Lie algebra.
\end{thm}
\begin{proof}
For any $u,v,w\in V$, we have
\begin{align*}
&[Tu,Tv,Tw]_{\tau}=^{\eqref{eq:induced1}}(\tau \circ T)(u)[Tu,Tw]_\mathfrak{g}+(\tau \circ T)(v)[Tw,Tu]_\mathfrak{g}+(\tau \circ T)(w)[Tu,Tv]_\mathfrak{g}.
\end{align*}
On the other hand, we have
\begin{align*}
& T\Big(\rho_{\tau}(Tu,Tv)w+\rho_{\tau}(Tv,Tw)u+\rho_{\tau}(Tw,Tu)v+\Theta_{\tau}(Tu,Tv,Tw)\Big)\\
&=T\Big((\tau \circ T)(u)\rho(Tv)w-(\tau \circ T)(v)\rho(Tu)w+(\tau \circ T)(v)\rho(Tw)u\\
&-(\tau \circ T)(w)\rho(Tv)u+(\tau \circ T)(w)\rho(Tu)v-(\tau \circ T)(u)\rho(Tw)v\\
&+(\tau \circ T)(u)\Theta(Tv,Tw)+(\tau \circ T)(v)\Theta(Tw,Tu)+(\tau \circ T)(w)\Theta(Tu,Tv)\Big).
\end{align*}
Thus  by Eq. \eqref{eq:tetatwist} we have
\begin{align*}
&[Tu,Tv,Tw]_{\tau}=T\Big(\rho_{\tau}(Tu,Tv)w+\rho_{\tau}(Tv,Tw)u+\rho_{\tau}(Tw,Tu)v+\Theta_{\tau}(Tu,Tv,Tw)\Big).
\end{align*}
Hence $T$ defines a $\Theta_{\tau}$-twisted $\mathcal{O}$-operator on the induced $3$-Lie algebra.
\end{proof}
\begin{defi}[\cite{Das1}]
An NS-Lie algebra is a vector space $A$ together with bilinear operations $\{\cdot,\cdot\},[\![\cdot,\cdot]\!] : A\times A \rightarrow A$
in which $[\![\cdot,\cdot]\!]$ is skew-symmetric and satisfying the following two identities
\begin{equation}\label{eq:nslie1}
\{\{x,y\},z\}-\{x,\{y,z\}\}-\{\{y,x\},z\}+\{y,\{x,z\}\}+\{[\![x,y]\!],z\}=0,
\end{equation}
\begin{equation}\label{eq:nslie2}
[\![x,[y ,z]_{\ast}]\!]+[\![y,[z ,x]_{\ast}]\!]+[\![z,[x ,y]_{\ast}]\!]+\{x,[\![y,z]\!]\}+\{y,[\![z,x]\!]\}+\{z,[\![x,y]\!]\}=0,
\end{equation}
for all $x,y,z\in A$, where $[x ,y]_{\ast}=\{x,y\}-\{y,x\}+[\![x,y]\!]$, for $x,y\in A$.
\end{defi}
\begin{pro}[\cite{Das1}]\label{dass}
Let $\Theta \in \mathfrak{C}_{CE}^{2}(\mathfrak{g},V)$ be a $2$-cocycle in the Chevalley-Eilenberg cohomology of $(\mathfrak{g},[\cdot,\cdot]_{\mathfrak{g}})$ with coefficients
in $(V,\rho)$. Let $T:V \rightarrow \mathfrak{g}$ be a $\Theta$-twisted $\mathcal{O}$-operator. Then
\begin{equation}
    \{u,v\}=\rho(Tu)v \quad \quad and \quad \quad  [\![u,v]\!]=\Theta(Tu,Tv),\;\forall u,v\in V,
\end{equation}
defines an NS-Lie algebra on $V$.
\end{pro}
\begin{thm}\label{indns}
Let $(A,\{\cdot,\cdot\},[\![\cdot,\cdot]\!])$ be an NS-Lie algebra and $\tau:A\rightarrow \mathbb{K}$ be a trace map, that is a linear form
satisfying \begin{equation}\label{eq:tauu}
    \tau([x ,y]_{\ast})=0,\; \forall x,y\in A. \end{equation} Define two  $3$-ary brackets by
\begin{equation}
\{x,y,z\}_{\tau}=\tau(x)\{y,z\}-\tau(y)\{x,z\},\end{equation}\label{eq:indu1}
\begin{equation}[\![x,y,z]\!]_{\tau}=\circlearrowleft_{x,y,z}\tau(x)[\![y,z]\!].\end{equation}\label{eq:indu2}
Then $(A,\{\cdot,\cdot,\cdot\}_{\tau},[\![\cdot,\cdot,\cdot]\!]_{\tau})$ is a $3$-NS-Lie algebra.
\end{thm}

\begin{proof}
Let $x,y,z\in A$. We have
\begin{align*}
\{x,y,z\}_{\tau}=\tau(x)\{y,z\}-\tau(y)\{x,z\}=-\Big(\tau(y)\{x,z\}-\tau(x)\{y,z\}\Big)=-\{y,x,z\}_{\tau},
\end{align*}
and since $[\![\cdot,\cdot]\!]$ is skew-symmetric, we have
\begin{align*}
&[\![x,y,z]\!]_{\tau}=\tau(x)[\![y,z]\!]+\tau(y)[\![z,x]\!]+\tau(z)[\![x,y]\!]\\
&=-\Big(\tau(y)[\![x,z]\!]+\tau(x)[\![z,y]\!]+\tau(z)[\![y,x]\!]\Big)=-[\![y,x,z]\!]_{\tau}.
\end{align*}
Similarly, we get $[\![x,y,z]\!]_{\tau}=-[\![x,z,y]\!]_{\tau}$, then $[\![\cdot,\cdot,\cdot]\!]_{\tau}$ is skew-symmetric.
For $x_i\in A,\;1 \leq i \leq 5$, we have
\begin{align*}
& \{\{x_1,x_2,x_3\}_{\tau}^{C},x_4,x_5\}_{\tau}+\{[\![x_1,x_2,x_3]\!]_{\tau},x_4,x_5\}_{\tau}-\{x_1,x_2,\{x_3,x_4,x_5\}_{\tau}\}_{\tau}\\
&-\{x_2,x_3,\{x_1,x_4,x_5\}_{\tau}\}_{\tau}-\{x_3,x_1,\{x_2,x_4,x_5\}_{\tau}\}_{\tau}\\
^{\eqref{eq:indu1}+\eqref{eq:indu2}}&=\tau(x_1)\Big(\tau(\{x_2,x_3\})\{x_4,x_5\}-\tau(x_4)\{\{x_2,x_3\},x_5\}\Big)\\
&-\tau(x_2)\Big(\tau(\{x_1,x_3\})\{x_4,x_5\}-\tau(x_4)\{\{x_1,x_3\},x_5\}\Big)\\
&+\tau(x_2)\Big(\tau(\{x_3,x_1\})\{x_4,x_5\}-\tau(x_4)\{\{x_3,x_1\},x_5\}\Big)\\
&-\tau(x_3)\Big(\tau(\{x_2,x_1\})\{x_4,x_5\}-\tau(x_4)\{\{x_2,x_1\},x_5\}\Big)\\
&+\tau(x_3)\Big(\tau(\{x_1,x_2\})\{x_4,x_5\}-\tau(x_4)\{\{x_1,x_2\},x_5\}\Big)\\
&-\tau(x_1)\Big(\tau(\{x_3,x_2\})\{x_4,x_5\}-\tau(x_4)\{\{x_3,x_2\},x_5\}\Big)\\
&+\tau(x_1)\Big(\tau([\![x_2,x_3]\!])\{x_4,x_5\}-\tau(x_4)\{[\![x_2,x_3]\!],x_5\}\Big)\\
&+\tau(x_2)\Big(\tau([\![x_3,x_1]\!])\{x_4,x_5\}-\tau(x_4)\{[\![x_3,x_1]\!],x_5\}\Big)\\
&+\tau(x_3)\Big(\tau([\![x_1,x_2]\!])\{x_4,x_5\}-\tau(x_4)\{[\![x_1,x_2]\!],x_5\}\Big)\\
&-\tau(x_3)\tau(x_1)\{x_2,\{x_4,x_5\}\}+\tau(x_3)\tau(x_2)\{x_1,\{x_4,x_5\}\}\\
&+\tau(x_4)\tau(x_1)\{x_2,\{x_3,x_5\}\}-\tau(x_4)\tau(x_2)\{x_1,\{x_3,x_5\}\}\\
&-\tau(x_1)\tau(x_2)\{x_3,\{x_4,x_5\}\}+\tau(x_1)\tau(x_3)\{x_2,\{x_4,x_5\}\}\\
&+\tau(x_4)\tau(x_2)\{x_3,\{x_1,x_5\}\}-\tau(x_4)\tau(x_3)\{x_2,\{x_1,x_5\}\}\\
&-\tau(x_2)\tau(x_3)\{x_1,\{x_4,x_5\}\}+\tau(x_2)\tau(x_1)\{x_3,\{x_4,x_5\}\}\\
&+\tau(x_4)\tau(x_3)\{x_1,\{x_2,x_5\}\}-\tau(x_4)\tau(x_1)\{x_3,\{x_2,x_5\}\}\\
&=\;\tau(x_1)\Big(\tau([x_2 ,x_3]_{\ast})\Big)\{x_4,x_5\}+\tau(x_2)\Big(\tau([x_3 ,x_1]_{\ast})\Big)\{x_4,x_5\}+\tau(x_3)\Big(\tau([x_1 ,x_2]_{\ast})\Big)\{x_4,x_5\}\\
&-\tau(x_1)\tau(x_4)\Big(\{\{x_2,x_3\},x_5\}-\{x_2,\{x_3,x_5\}\}-\{\{x_3,x_2\},x_5\}+\{x_3,\{x_2,x_5\}\}+\{[\![x_2,x_3]\!],x_5\}\Big)\\
&-\tau(x_2)\tau(x_4)\Big(\{\{x_3,x_1\},x_5\}-\{x_3,\{x_1,x_5\}\}-\{\{x_1,x_3\},x_5\}+\{x_1,\{x_3,x_5\}\}+\{[\![x_3,x_1]\!],x_5\}\Big)\\
&-\tau(x_3)\tau(x_4)\Big(\{\{x_1,x_2\},x_5\}-\{x_1,\{x_2,x_5\}\}-\{\{x_2,x_1\},x_5\}+\{x_2,\{x_1,x_5\}\}+\{[\![x_1,x_2]\!],x_5\}\Big)\\
^{\eqref{eq:nslie1}+\eqref{eq:tauu}} &=\; 0.
\end{align*}
Then Eq. \eqref{eq:NS2} holds. Similarly, Eqs.\eqref{eq:NS1}, \eqref{eq:NS3} holds by Eqs.\eqref{eq:nslie1}-\eqref{eq:tauu}. This completes the proof.
\end{proof}
Let $(A,[\cdot,\cdot])$ be a Lie algebra and $(V,\rho)$ be a representation. We say that we have a Lie-Rep pair and refer to it with the tuple $(A,[\cdot,\cdot],V,\rho)$. We use similar notation for a 3-Lie algebra together with a representation.  The following proposition gives us two different ways to construct the same $3$-NS-Lie algebra structure on $V$ using a trace map or a $\Theta$-twisted
$\mathcal{O}$-operator. In particular, we obtain a  $3$-NS-Lie algebra on $A$ using the adjoint representation.
\begin{pro}
We have the following diagram
\begin{equation*}\label{Diagramme1}\begin{split}
\xymatrix{
\ar[rr] \mbox{ Lie-Rep pair $(A,[\cdot,\cdot],V,\rho)$ }\ar[d]_{\mbox{\small $\Theta$-twisted
$\mathcal{O}$-operator}}\ar[rr]^{\mbox{\quad\quad $\tau$\quad\quad }}
                && \mbox{ $3$-Lie-Rep pair 
            $(A,[\cdot,\cdot,\cdot]_{\tau},V,\rho_{\tau})$ }\ar[d]_{\mbox{\small $\Theta_{\tau}$-twisted
$\mathcal{O}$-operator}}\\
\ar[rr] \mbox{ NS-Lie algebra $(V,\{\cdot,\cdot\},[\![\cdot,\cdot]\!])$}\ar@<-1ex>[u]_{\small \mbox{\small $[\cdot,\cdot]_{\ast}$ }}\ar[rr]^{\mbox{\small $\tau \circ T$ \quad\quad}}
                && \mbox{ 3-NS-Lie algebra $(V,\{\cdot,\cdot,\cdot\}_{\tau},[\![\cdot,\cdot,\cdot]\!]_{\tau})$}\ar@<-1ex>[u]_{\mbox{\small $[\cdot,\cdot,\cdot]_{\ast}$}}}\end{split}
\end{equation*}
\end{pro}
\begin{proof}
Let $\tau:A \rightarrow \mathbb{K}$ be a trace map on $(A,[\cdot,\cdot])$, $(V,\rho)$ be a representation and $\Theta \in \mathfrak{C}_{CE}^{2}(\mathfrak{g},V)$ be a $2$-cocycle. Then $(A,[\cdot,\cdot,\cdot]_{\tau})$ is an induced $3$-Lie algebra, $(V,\rho_{\tau})$ is a representation of the induced $3$-Lie algebra and, by Proposition \ref{eq:induced3}, $\Theta_{\tau}$ is a $2$-cocycle of the induced $3$-Lie algebra.  
Let  $T:V \rightarrow A$ be  a $\Theta_{\tau}$-twisted $\mathcal{O}$-operator on $(A,[\cdot,\cdot,\cdot]_{\tau})$ with respect to  $(V,\rho_{\tau})$. Then, for any $u,v,w\in V$ and according to  Proposition \ref{eq:onV}, the brackets 
\begin{equation*}
    \{u,v,w\}_1=\rho_{\tau}(Tu,Tv)w,\quad \quad \text{and} \quad  \quad [\![u,v,w]\!]_{1}=\Theta_{\tau}(Tu,Tv,Tw),
\end{equation*}
defines a $3$-NS-Lie algebra on $V$. Also, one can construct a $3$-NS-Lie algebra on $V$ by another method. By Proposition \ref{dass}, we  define an NS-Lie algebra on $V$. By Theorem \ref{indns}, let $\tau ':V \rightarrow \mathbb{K}$ be a linear form such that $\tau'([u,v]_*)=0,\; \forall u,v\in V$. Then
\begin{equation*}
    \{u,v,w\}_2=\tau'(u)\{v,w\}-\tau'(v)\{u,w\}=\tau'(u)\rho(Tv)w-\tau'(v)\rho(Tu)w,\;\;\forall u,v,w\in V,
\end{equation*}
and
\begin{equation*}
   [\![u,v,w]\!]_{2}= \circlearrowleft_{u,v,w}\tau'(u)[\![v,w]\!]=\circlearrowleft_{u,v,w}\tau'(u)\Theta(Tv,Tw),\;\;\forall u,v,w\in V,
\end{equation*}
define a $3$-NS-Lie algebra on $V$. By a direct computation, we have
\begin{align*}
    &\{u,v,w\}_2-\{u,v,w\}_1=\tau'(u)\rho(Tv)w-\tau'(v)\rho(Tu)w-\Big((\tau \circ T)(u)\rho(Tv)w-(\tau \circ T)(v)\rho(Tu)w\Big)\\
   &=\Big(\tau'(u)-(\tau \circ T)(u)\Big) \rho(Tv)w-\Big(\tau'(v)-(\tau \circ T)(v)\Big) \rho(Tu)w,\;\;\forall u,v,w\in V,
\end{align*}
and 
\begin{align*}
   &[\![u,v,w]\!]_{2}-[\![u,v,w]\!]_{1}= \circlearrowleft_{u,v,w}\tau'(u)\Theta(Tv,Tw)-\circlearrowleft_{u,v,w}(\tau \circ T)(u)\Theta(Tv,Tw)\\
  & =\circlearrowleft_{u,v,w}\Big(\tau'(u)-(\tau \circ T)(u)\Big)\Theta(Tv,Tw),\;\;\forall u,v,w\in V.
\end{align*}
If we take $\tau'=\tau \circ T$, then we get the same $3$-NS-Lie algebra structure on $V$.\\

In particular, let $(V,\rho)=(A,ad)$, the adjoint representation of $(A,[\cdot,\cdot])$, and consider $ad_{\tau}:\wedge^{2}A \rightarrow gl(A)$ defined by 
\begin{equation*}
    ad_{\tau}(x,y)z=\tau(x)ad(y)z-\tau(y)ad(x)z=\tau(x)[y,z]-\tau(y)[x,z],\;\forall x,y,z\in A.
\end{equation*}
Then, one can follow the previous computations to  obtain a $3$-NS-Lie algebra structure  on $A$ using two different methods.
\end{proof} 
%%%%%%%%%%%%%%%%%%%%%%%%%%%%%%%

\end{document}